\documentclass[11pt,reqno,letterpaper]{amsart}

\usepackage{amsmath,amsthm,amssymb,mathtools,bbm,enumerate}
\usepackage{graphicx}
\usepackage{bookmark}
\usepackage{hyperref}
\hypersetup{pdfstartview={FitH}}
\usepackage[british]{babel}

\theoremstyle{plain}
\newtheorem{theorem}{Theorem}[section]
\newtheorem{proposition}[theorem]{Proposition}
\newtheorem{corollary}[theorem]{Corollary}
\newtheorem{lemma}[theorem]{Lemma}

\numberwithin{equation}{section}

\renewcommand{\leq}{\leqslant}
\renewcommand{\le}{\leqslant}
\renewcommand{\geq}{\geqslant}
\renewcommand{\ge}{\geqslant}

\newcommand{\R}{\mathbb{R}}
\newcommand{\T}{\mathbb{T}}
\newcommand{\N}{\mathbb{N}}
\newcommand{\Z}{\mathbb{Z}}
\newcommand{\1}{\mathbbm{1}}
\newcommand{\dd}{\,\textup{d}}

\begin{document}

\title[Estimates for $\textup{L}^p$ variants of Gowers norms]{Estimates for $\textup{L}^p$ variants of Gowers norms}

\author[V. Kova\v{c}]{Vjekoslav Kova\v{c}}
\address[V.\,K.]{University of Zagreb Faculty of Science, Department of Mathematics, Bijeni\v{c}ka cesta 30, 10000 Zagreb, Croatia}
\email{vjekovac@math.hr}

\author[K. M. Rogers]{Keith M. Rogers}
\address[K.\,R.]{Instituto de Ciencias Matem\'{a}ticas CSIC-UAM-UC3M-UCM, 28049 Madrid, Spain}
\email{keith.rogers@icmat.es}

\subjclass[2020]{Primary 26D15; 
Secondary 11B30, 
43A70, 
94A17} 

\keywords{Gowers norm, sharp estimate, near-extremiser, Shannon entropy}

\begin{abstract}
Motivated by a log-convexity question of Bennett and Tao, we consider Gowers-type functionals defined by $\textup{L}^p$ norms of multiple autocorrelations. We prove sharp bounds in terms of $\textup{L}^q$ norms and characterise the near-extremisers on Euclidean spaces as well as locally compact abelian groups. We also establish two families of degree-lowering inequalities and study their near-extremisers. As a byproduct of this broader theory, we show that the constant in the log-convexity estimate for Gowers norms is strictly less than unity, confirming the aforementioned conjecture of Bennett and Tao. Finally, we characterise the $p$ for which the Gowers-type functionals necessarily satisfy the triangle inequality on nonnegative measurable functions.
\end{abstract}

\maketitle

\tableofcontents

\newpage
\section{Introduction}

\subsection{Gowers norms, their estimates and log-convexity}

Gowers introduced the uniformity norms $\|\cdot\|_{\textup{U}^d(G)}$ in his quantitative proof of Szemer\'edi's Theorem~\cite{Gowers2001}. Later they became an indispensable tool in numerous problems in arithmetic combinatorics; see for example the book by Tao and Vu \cite{TaoVu}. Host and Kra introduced ergodic-theoretic variants of these quantities \cite{HostKra05}. Their variants on Euclidean spaces naturally arise in connection with problems in geometric measure theory; see for example \cite{CMP15} and \cite{DK22}.

Throughout, $(G,+)$ is a second-countable locally compact abelian (LCA) group equipped with a Haar measure $\mu$, which is unique up to a positive scalar multiple. The second-countability assumption is not necessary for much of the article, but it is convenient and also assumed in the work of Eisner and Tao \cite{EisnerTao}, on which we rely. In particular, $G$ is metrisable and $\sigma$-compact, which makes its Haar measure $\sigma$-finite and allows straightforward applications of Fubini's theorem.
For any complex-valued function $f$ on $G$, let $\Delta_h f$ be its \emph{multiplicative difference}, defined by
\[ \Delta_h f(x) := \overline{f(x)} f(x+h) \]
for $h,x\in G$.
For measurable $f$ and an integer $d\geq1$, the \emph{Gowers uniformity (semi)norm} is defined by
\[ \|f\|_{\textup{U}^d(G)} := \biggl( \int_{G^{d+1}} \Delta_{h_1} \cdots \Delta_{h_d} f(x)\, \dd\mu(x)\dd\mu(h_1)\cdots\dd\mu(h_d) \biggr)^{1/2^{d}} \]
whenever the integral exists.
For $d\geq2$, this is in fact a norm on the space of measurable functions $f$ satisfying $\|f\|_{\textup{U}^d(G)}<\infty$, modulo equality almost everywhere \cite{TaoVu,BennettTao}. For $d=1$ we obtain a simple seminorm
\[ \|f\|_{\textup{U}^1(G)} = \Bigl| \int_{G} f \dd\mu \Bigr|. \]
If $\mathcal{C}$ denotes the operator of complex conjugation, i.e., $\mathcal{C}z:=\overline{z}$ so that $\mathcal{C}^2 z=z$, then the previous definition can be written explicitly as
\begin{align*}
\|f\|_{\textup{U}^d(G)} = \biggl( \int_{G^{d+1}} \prod_{(\omega_1,\ldots,\omega_d)\in\{0,1\}^d}
\mathcal{C}^{d-(\omega_1+\cdots+\omega_d)} f\bigl(x+\omega_1 h_1+\cdots+\omega_d h_d\bigr) & \\[-4mm]
\dd\mu(x)\dd\mu(h_1)\cdots&\dd\mu(h_d) \biggr)^{1/2^{d}}.
\end{align*}

One line of research concerns sharp $\textup{L}^p$ estimates for the Gowers norms. Namely, on every LCA group $G$ we have the inequality
\begin{equation}\label{eq:UGbyLp}
\|f\|_{\textup{U}^d(G)} \leq
\|f\|_{\textup{L}^{2^d/(d+1)}(G)},
\end{equation}
which follows easily from H\"{o}lder's inequality and mathematical induction; see for example \cite[Eq.\,(5)]{EisnerTao} or the proof of Theorem~\ref{thm:Lpest}(a) below.
Eisner and Tao \cite[Thm.\,1.10]{EisnerTao} studied near-extremisers of \eqref{eq:UGbyLp}: if $d\geq2$, $\varepsilon>0$ is sufficiently small in terms of $d$ and $f\in\textup{L}^{2^d/(d+1)}(G)$ is not equal to zero $\mu$-a.e.\@ and satisfies
\[ \|f\|_{\textup{U}^d(G)} \geq (1-\varepsilon) \|f\|_{\textup{L}^{2^d/(d+1)}(G)}, \]
then there exist a compact open subgroup $H$ of $G$, an element $x_0\in G$ and a polynomial $P\colon H\to\R/\Z$ of degree at most $d-1$ such that
\begin{equation}\label{eq:ET_near}
\Bigl\| f - \frac{\|f\|_{\textup{L}^{2^d/(d+1)}(G)}}{\mu(H)^{(d+1)/2^d}}e^{2\pi i P(\cdot-x_0)} \1_{x_0+H} \Bigr\|_{\textup{L}^{2^d/(d+1)}(G)}
= o^{\varepsilon\to0}_{d}(1) \|f\|_{\textup{L}^{2^d/(d+1)}(G)}.
\end{equation}
Since $G=\R^n$ has no compact open subgroups, the corresponding best constant in \eqref{eq:UGbyLp} is strictly less than $1$ when $d\geq2$ and the sharp Euclidean inequality \cite[Thm.\,1.12]{EisnerTao} reads
\begin{equation}\label{eq:UonRn}
\|f\|_{\textup{U}^d(\R^n)} \leq \Bigl( \frac{2^{2d}}{(d+1)^{d+1}} \Bigr)^{n/2^{d+1}} \|f\|_{\textup{L}^{2^d/(d+1)}(\R^n)}.
\end{equation}
Eisner and Tao also characterised the cases of equality in \eqref{eq:UonRn}, while Christ \cite{ChristYoung2019} and Neuman \cite{Neuman2020} obtained a characterisation of the near-extremisers.

Another line of research studies log-convexity properties of the Gowers norms. Inequality \eqref{eq:UGbyLp} suggests that they might imitate the $\textup{L}^p$-norms with exponents $2^d/(d+1)$. Noting that $\theta(d):=d/(3d-2)$ satisfies
\[ \frac{d+1}{2^d} =\theta(d)\frac{d}{2^{d-1}} + \big(1-\theta(d)\big)\frac{d+2}{2^{d+1}}, \]
as a special case of their adjoint Loomis--Whitney inequality, Bennett and Tao \cite[Cor.\,2.4]{BennettTao} showed that
\begin{equation}\label{eq:BTlogcv}
\|f\|_{\textup{U}^d(G)} \leq
\|f\|_{\textup{U}^{d-1}(G)}^{\theta(d)}
\|f\|_{\textup{U}^{d+1}(G)}^{1-\theta(d)}
\end{equation}
whenever $f\colon G\to[0,\infty)$ is measurable.
In the setting of discrete groups $G$, this was previously discovered by Shkredov \cite[Prop.\,35]{Shkredov2014} (for indicator functions) and Manners \cite[Prop.\,2.1]{Manners2017} (for $d=2$ and finitely supported functions). Bennett and Tao asked \cite[Quest.\,10.1]{BennettTao} whether the constant can be lowered when $G=\R$ and $d=2$, i.e., whether there exists an $\varepsilon>0$ such that
\[ \|f\|_{\textup{U}^2(\R)} \leq (1-\varepsilon) \|f\|_{\textup{U}^{1}(\R)}^{1/2} \|f\|_{\textup{U}^{3}(\R)}^{1/2} \]
whenever $f\colon \R\to[0,\infty)$ is measurable. As suggestive evidence, they showed that the corresponding equality cannot be attained \cite[Prop.\,2.5]{BennettTao}, i.e., that one always has the strict inequality
\[ \|f\|_{\textup{U}^2(\R)} < \|f\|_{\textup{U}^{1}(\R)}^{1/2} \|f\|_{\textup{U}^{3}(\R)}^{1/2} \]
whenever all three Gowers norms are positive and finite.
The same question also naturally arose during the Q\&A session following Tao's seminar talk at the IAS in 2023~\cite{Tao_IAS}.

The following theorem answers the question in greater generality.

\begin{theorem}\label{thm:mainlog}
For every integer $d\geq2$, there exists $\varepsilon_d\in(0,1)$ such that
\begin{equation}\label{eq:mainlog}
\|f\|_{\textup{U}^d(G)} \leq (1-\varepsilon_d) \|f\|_{\textup{U}^{d-1}(G)}^{\theta(d)} \|f\|_{\textup{U}^{d+1}(G)}^{1-\theta(d)}
\end{equation}
whenever $f\colon G\to[0,\infty)$ is measurable and $G$ has no compact open subgroup.
\end{theorem}

Note that $\varepsilon_d$ is uniform over all second-countable LCA groups $G$ without a compact open subgroup. If $H$ is a compact open subgroup of $G$, then choosing $f=\1_H$ and observing
\[ \|\1_H\|_{\textup{U}^d(G)}=\mu(H)^{(d+1)/2^d} \in(0,\infty), \]
we obtain equality in \eqref{eq:BTlogcv}, so the absence of compact open subgroups is necessary as well as sufficient.
It should also be possible to characterise the near-extremisers of \eqref{eq:BTlogcv} on general LCA groups $G$. We do not pursue this here because our proof yields closeness in a different metric, rather than a clean statement analogous to \eqref{eq:ET_near}.

Theorem~\ref{thm:mainlog} applies to the particular case $G=\R$ and confirms the conjecture of Bennett and Tao.
Note that $\varepsilon_d$ cannot be uniform in $d$. Namely, substituting the Gaussian $f(x)=e^{-\pi x^2}$ yields that the best constant in \eqref{eq:mainlog} is at least
\begin{equation}\label{eq:logGauss}
2^{-d(d-1)/((3d-2)2^{d+1})},
\end{equation}
which converges to $1$ as $d\to\infty$. Moreover, it is easy to see that Gaussians are not extremisers of \eqref{eq:mainlog}. Plugging in the indicator $f=\1_I$ of any bounded interval $I\subset\R$ of positive length gives the ratio
\[ \frac{\|\1_I\|_{\textup{U}^d(\R)} }{ \|\1_I\|_{\textup{U}^{d-1}(\R)}^{d/(3d-2)} \|\1_I\|_{\textup{U}^{d+1}(\R)}^{2(d-1)/(3d-2)} }
= \Bigl( \frac{2\,d!\,(d+2)^{d-1}}{(d+1)^{2d-1}} \Bigr)^{1/((3d-2)2^d)}, \]
which is larger than \eqref{eq:logGauss} for every $d\geq2$; also see \cite[Thm.\,6.1]{BennettTao} for more general inequalities for which Gaussians are not extremisers either. We are not in a position to find the best constant; our proof of \eqref{eq:mainlog} could give effective but poor bounds on~$\varepsilon_d$.

\subsection{Gowers-type functionals on LCA groups}

Bennett and Tao \cite[Sec.\,2]{BennettTao} considered a certain \emph{multiple autocorrelation function} $R_d f$. For every integer $d\geq1$ and measurable function $f$ on $G$, define
\begin{align*}
R_d f (h_1,\ldots,h_d)
& := \int_{G} \Delta_{h_1} \cdots \Delta_{h_d} f(x)\, \dd\mu(x) \\
& \,= \int_{G} \prod_{(\omega_1,\ldots,\omega_d)\in\{0,1\}^d} \mathcal{C}^{d-(\omega_1+\cdots+\omega_d)} f\bigl(x+\omega_1 h_1+\cdots+\omega_d h_d\bigr)\, \dd\mu(x)
\end{align*}
whenever the integral exists.
We go one step further and, for $p\in(0,\infty]$, define the Gowers-type functional by
\begin{equation}\label{eq:defUdp}
[f]_{\textup{U}^{d,p}(G)} := \|R_d f\|_{\textup{L}^p(G^d)}^{1/2^d}.
\end{equation}
This quantity, with a non-integer value of $p$, will play a crucial role in our proof of Theorem~\ref{thm:mainlog}. For integer values of $p$,
Shkredov \cite{Shkredov23} studied such expressions in the context of higher additive energies and named them the $E_{p,2,\ldots,2}$ norms; see also \cite{Hatami}.
In particular, if $d\geq2$ and $p$ is an even positive integer, then the functionals $f\mapsto[f]_{\textup{U}^{d,p}(G)}$ are norms, while, if $p$ is an arbitrary positive integer, the same is true for $f\mapsto[|f|]_{\textup{U}^{d,p}(G)}$.
See the precise formulations in Corollary \ref{cor:Udpnorms} and a sketch of the proof in Appendix \ref{sec:Udp_are_norms}.

Somewhat surprisingly, if $d\geq2$ and $p>0$ is not an integer, then $f\mapsto[f]_{\textup{U}^{d,p}(G)}$ is not necessarily subadditive on nonnegative functions. Namely, already in the case of the torus $\T=\R/\Z$, for every $d\geq2$ and every positive non-integer $p$ there exist nonnegative trigonometric polynomials $f_1$ and $f_2$ such that
\begin{equation}\label{eq:triangle_fails}
[f_1+f_2]_{\textup{U}^{d,p}(\T)} > [f_1]_{\textup{U}^{d,p}(\T)} + [f_2]_{\textup{U}^{d,p}(\T)}.
\end{equation}
 The counterexample, which is strongly inspired by the constructions that resolved the Hardy--Littlewood majorant problem \cite{Boas62,Bachelis73,GreenRuzsa}, is presented in Appendix \ref{sec:Udp_not_norms}.

In fact, Bennett and Tao already used $[f]_{\textup{U}^{d,p}(G)}$ with non-integer parameter $p=d/(2d-1)$ in their proof of \eqref{eq:BTlogcv}, but we are unaware of a systematic study of \eqref{eq:defUdp} for real values of $p>0$. We provide such a study; the proof of Theorem~\ref{thm:mainlog} will then come as a byproduct.

Our first result on the quantities $[\,\cdot\,]_{\textup{U}^{d,p}(G)}$ is a generalisation of \eqref{eq:UGbyLp} and \eqref{eq:ET_near}.
Throughout the main body of the article, we consider only nonnegative functions $f$.

\begin{theorem}\label{thm:Lpest}
\begin{enumerate}[(a)]
\item Let $d\geq2$ and $d/(2^d-1)\leq p<\infty$. Then every measurable function $f\colon G\to[0,\infty)$ satisfies
\[ [f]_{\textup{U}^{d,p}(G)}
 \leq\|f\|_{\textup{L}^{2^d p/(d+p)}(G)}. \]
\item Let $d\geq2$, $d/(2^d-1)< p< \infty$, suppose that $\varepsilon$ is sufficiently small in terms of $d$ and $p$ and let $f\in\textup{L}^{2^d p/(d+p)}(G)$ be nonnegative, not equal to zero $\mu$-a.e.\@ and such that
\[ [f]_{\textup{U}^{d,p}(G)} \geq (1-\varepsilon) \|f\|_{\textup{L}^{2^d p/(d+p)}(G)}. \]
Then there exist a compact open subgroup $H$ of $G$ and an element $x_0\in G$ such that
\[ \Bigl\| f - \frac{\|f\|_{\textup{L}^{2^d p/(d+p)}(G)}}{\mu(H)^{(d+p)/(2^d p)}}\1_{x_0+H} \Bigr\|_{\textup{L}^{2^d p/(d+p)}(G)}
= o^{\varepsilon\to0}_{d,p}(1) \|f\|_{\textup{L}^{2^d p/(d+p)}(G)}. \]
\end{enumerate}
\end{theorem}

Next, recall from \eqref{eq:UonRn} that the best constant in the corresponding estimate on Euclidean groups $G=\R^n$ is known. The same holds in greater generality, and near-extremisers can again be characterised.

\begin{theorem}\label{thm:sharpRn}
Let $d\geq2$ and $n\geq1$ be integers, and let $p\in[1,\infty)$.
\begin{enumerate}[(a)]
\item Then every measurable function $f\colon \R^n\to[0,\infty)$ satisfies
\[ [f]_{\textup{U}^{d,p}(\R^n)} \leq \Bigl( \frac{2^{2d} p^p}{(d+p)^{d+p}} \Bigr)^{n/(2^{d+1}p)} \|f\|_{\textup{L}^{2^d p/(d+p)}(\R^n)} . \]
The above constant is optimal. Among nonnegative functions $f\in\textup{L}^{2^d p/(d+p)}(\R^n)$ that are not equal to $0$ a.e., equality is attained precisely when $f$ agrees with a Gaussian up to a null set, i.e.,
\[ f(x)=c\exp\bigl(-(x-x_0)\cdot M(x-x_0)\bigr) \]
for almost every $x\in\R^n$, where $c>0$, $x_0\in\R^n$ and $M$ is a real symmetric positive-definite $n\times n$ matrix.
\item Moreover, if $\varepsilon$ is sufficiently small in terms of $d,p,n$, and $f\in\textup{L}^{2^d p/(d+p)}(\R^n)$ is nonnegative, not identically zero modulo null sets and such that
\[ [f]_{\textup{U}^{d,p}(\R^n)} \geq (1-\varepsilon) \Bigl( \frac{2^{2d} p^p}{(d+p)^{d+p}} \Bigr)^{n/(2^{d+1}p)} \|f\|_{\textup{L}^{2^d p/(d+p)}(\R^n)} , \]
then there exists a Gaussian $g$ such that
\[ \|g\|_{\textup{L}^{2^d p/(d+p)}(\R^n)}=\|f\|_{\textup{L}^{2^d p/(d+p)}(\R^n)} \]
and
\[ \|f - g\|_{\textup{L}^{2^d p/(d+p)}(\R^n)} = o^{\varepsilon\to0}_{d,p,n}(1)\,\|f\|_{\textup{L}^{2^d p/(d+p)}(\R^n)}. \]
\end{enumerate}
\end{theorem}

Finally, we provide two families of degree-lowering inequalities which cannot be formulated within the class of the usual Gowers norms.
In the following theorem, the real numbers $p,p_1,\ldots,p_d$ and $\theta_1,\ldots,\theta_d$ are related by
\begin{equation}\label{relation} \frac{1}{d-1}\Bigl(1-\frac{1}{p}\Bigr) = \theta_i\Bigl(1-\frac{2}{p_i}\Bigr) \end{equation}
for $i=1,\ldots,d$. The group $G$ is again a second-countable LCA group.

\begin{theorem}\label{thm:diffUdp}
\begin{enumerate}[(a)]
\item Let $d\ge 2$, and let $0<p\leq1$, $\theta_i>0$ satisfy \eqref{relation} and
$\sum_{i=1}^{d}\theta_i=1$. Then every measurable
$f\colon G\to[0,\infty)$ satisfies
\[ [f]_{\textup{U}^{d,p}(G)} \leq \prod_{i=1}^d [f]_{\textup{U}^{d-1,p_i}(G)}^{\theta_i}. \]
Let $d\ge 2$, let $1<p<\infty$ and let $\theta_1>0$, $\theta_i<0$ for $i\geq2$ satisfy \eqref{relation} and $\sum_{i=1}^{d}\theta_i=1$. Then every measurable $f\colon G\to[0,\infty)$ satisfies the reverse inequality
\[ [f]_{\textup{U}^{d,p}(G)} \geq \prod_{i=1}^d [f]_{\textup{U}^{d-1,p_i}(G)}^{\theta_i}. \]
In order to make the latter inequality meaningful, we also assume that $f$ is not equal to zero almost everywhere and that all involved Gowers-type functionals are finite.
\item Suppose that the parameters $p\neq1$, $p_i$ and $\theta_i$ obey one of the two sets of assumptions in part~(a) and write $\mathbf{p}=(p_1,\ldots,p_d)$ and $\boldsymbol{\theta}=(\theta_1,\ldots,\theta_d)$. Let $f\colon G\to[0,\infty)$ be measurable with all involved Gowers-type functionals finite and nonzero. If, for a sufficiently small $\varepsilon>0$, the corresponding inequality is nearly extremised in the sense that
\[ [f]_{\textup{U}^{d,p}(G)} \geq (1-\varepsilon) \prod_{i=1}^d [f]_{\textup{U}^{d-1,p_i}(G)}^{\theta_i} \]
or
\[ \prod_{i=1}^d [f]_{\textup{U}^{d-1,p_i}(G)}^{\theta_i} \geq (1-\varepsilon) [f]_{\textup{U}^{d,p}(G)}, \]
then there are measurable $r_1,\ldots,r_d:G\to [0,\infty)$, with integral 1, such that
\[ \Bigl\| (R_d f)^p - \|R_d f\|_{\textup{L}^p(G^d)}^p (r_1\otimes\cdots\otimes r_d) \Bigr\|_{\textup{L}^1(G^d)}
= o^{\varepsilon\to0}_{d,p,\mathbf{p},\boldsymbol{\theta}}(1) \|R_d f\|_{\textup{L}^p(G^d)}^p \]
 and
\[ \Bigl\| (R_{d-1} f)^{p_i} -\|R_{d-1} f\|_{\textup{L}^{p_i}(G^{d-1})}^{p_i} \mathop{\bigotimes}_{j\ne i}r_j \Bigr\|_{\textup{L}^1(G^{d-1})}
= o^{\varepsilon\to0}_{d,p,\mathbf{p},\boldsymbol{\theta}}(1) \|R_{d-1} f\|_{\textup{L}^{p_i}(G^{d-1})}^{p_i} \]
for $i=1,\ldots,d$, respectively.
\end{enumerate}
\end{theorem}

In order to prove Theorem~\ref{thm:mainlog} we will assume for a contradiction that the constant~$1$ in the log-convexity estimate~\eqref{eq:BTlogcv} is optimal. Then, using Theorem~\ref{thm:diffUdp}, we prove the log-convexity estimate \eqref{eq:BTlogcv} in two complementary ways. A near-extremiser $f$ of \eqref{eq:BTlogcv} would also have to be a near-extremiser for both sets of inequalities involved in the two different proofs. This rigidity would force $R_{d-1} f$ and $R_d f$ to approximate indicator functions of sets. Recalling the autocorrelation structure, with some additional work this would contradict Fournier's sharpened Young convolution inequality \cite[Thm.\,1]{Fournier}.

\subsection{Gowers-type functionals on discrete cubes}

For positive integer exponents $p$, the functionals $[\,\cdot\,]_{\textup{U}^{d,p}}$ also yield sharp dimension-free estimates on discrete cubes. Namely, equip $\Z^m$ with the counting measure and fix integers $d\geq2$, $p\in\N$ and $n\geq2$. Let $q_{d,p,n}$ denote the supremum of numbers $q>0$ such that
\begin{equation}\label{eq:ell_critical}
[f]_{\textup{U}^{d,p}(\Z^m)} \leq\|f\|_{\ell^q(\Z^m)}
\end{equation}
for every dimension $m$ and every nonnegative $f$ supported in $\{0,1,\ldots,n-1\}^m$. It will turn out that this supremum is a maximum.
To relate this to a combinatorial problem, for $A\subseteq\{0,1,\ldots,n-1\}^m$ define an additive energy of $A$ as
\[ \mathcal{P}_{d,p}(A)
:= [\1_A]_{\textup{U}^{d,p}(\Z^m)}^{2^d p}
= \sum_{h_1,\ldots,h_d\in\Z^m} \bigl(R_d\1_A(h_1,\ldots,h_d)\bigr)^p. \]
By multiplying out the summand as $p$ repeated integrals (see \eqref{eq:rewriteUdp2}), $\mathcal{P}_{d,p}(A)$ counts the tuples $(x_1,\ldots,x_p,h_1,\ldots,h_d)$ for which
\[ x_j+\omega_1h_1+\cdots+\omega_dh_d\in A
\quad\text{for every }1\leq j\leq p \text{ and }(\omega_1,\ldots,\omega_d)\in\{0,1\}^d. \]
The particular case $p=1$ of these energies was introduced by Shkredov \cite{Shkredov2014} and studied further on discrete cubes by Beker, Crmari\'{c} and the first author \cite{BCK25}.
Let $t_{d,p,n}$ be the infimum of numbers $t>0$ such that
\begin{equation}\label{eq:dcset}
\mathcal{P}_{d,p}(A)\leq|A|^t
\end{equation}
for every $m$ and every such $A$. Again, we will show below that this infimum is, in fact, a minimum.
Estimates of the form \eqref{eq:dcset} were also studied in \cite{BCK25}, while their first instance appeared in the work of Kane and Tao \cite{KT17}.
The product principle in Proposition \ref{prop:dcprod} below shows that the two extremal exponents are attained and satisfy
\begin{equation}\label{eq:dcrel}
q_{d,p,n}t_{d,p,n} = p 2^d,
\end{equation}
so the problem of estimating $q_{d,p,n}$ and the problem of estimating $t_{d,p,n}$ are closely related.
We obtain the following extensions of the three principal results of \cite{BCK25}. Those results can be viewed as special cases for $p=1$.

\begin{theorem}\label{thm:dcres}
For $d\geq2$, $p\in\N$ and $n\geq2$, the following hold:
\begin{enumerate}[(a)]
\item On binary cubes,
\[ t_{d,p,2}=\log_2(2^p+2d), \quad q_{d,p,2}=\frac{2^d p}{\log_2(2^p+2d)}. \]
\item For fixed $d$ and $p$,
\begin{equation}\label{eq:dclargen}
\begin{aligned}
d+p -\bigl(1+o_{d,p}^{n\to\infty}(1)\bigr) \frac{(d+p)\log_2(d+p)-p\log_2p-2d}{2\log_2n} & \\
\leq t_{d,p,n} \leq d+p-\frac{c}{\log_2n} &,
\end{aligned}
\end{equation}
as $n\to\infty$, where $c>0$ is an absolute constant.
\item For fixed $n$ and $p$,
\[ t_{d,p,n} = \frac{(n-1)\log_2(2d)-\log_2((n-1)!)}{\operatorname{H}(\operatorname{B}(n-1,1/2))} + o_{n,p}^{d\to\infty}(1), \]
as $d\to\infty$, where
\[ \operatorname{H}\biggl(\operatorname{B}\Bigl(n-1,\frac{1}{2}\Bigr)\biggr)
= -\sum_{j=0}^{n-1}\frac{\binom{n-1}{j}}{2^{n-1}} \log_2\frac{\binom{n-1}{j}}{2^{n-1}} \]
is the Shannon entropy of the symmetric binomial random variable with $n-1$ trials.
\end{enumerate}
\end{theorem}

Surprisingly, the parameter $p$ plays no role in the main asymptotic term in part (c).
The main reason why Theorem~\ref{thm:dcres} is stated only for positive integers $p$ is that its proof will use the triangle inequality for $[|\cdot|]_{\textup{U}^{d,p}}$ and also a slightly stronger Gowers--Cauchy--Schwarz inequality from Lemma~\ref{lm:UdpGCS}. However, one might expect the same asymptotic properties of the analogous quantities $t_{d,p,n}$ and $q_{d,p,n}$ to be retained for general $p>0$.


\subsection{Outline of the article}

In Section \ref{sec:deglow} we deduce the degree-lowering inequalities of Theorem~\ref{thm:diffUdp}(a) from the adjoint Loomis--Whitney inequalities of Bennett and Tao. Section \ref{sec:Lebest} establishes the Lebesgue estimates on LCA groups and describes their near-extremisers, while Section \ref{sec:Euclid} proves the sharp Euclidean inequality and the fact that its near-extremisers are necessarily close to Gaussians. Section \ref{sec:lowering_extremisers} investigates the near-extremisers in Theorem~\ref{thm:diffUdp}(b) and Section \ref{sec:log-convexity} builds on those considerations to prove Theorem~\ref{thm:mainlog}. Section \ref{sec:realobs} explains why we do not expect that Theorem~\ref{thm:mainlog} can be proved by merely observing that the one-dimensional autocorrelation of $f$ must be close to an indicator function of a measurable set. Section \ref{sec:dcube} studies the sharp estimates for $[\,\cdot\,]_{\textup{U}^{d,p}(G)}$ on discrete cubes. Finally, Appendix~\ref{sec:Udp_are_norms} sketches the proof that the triangle inequality holds for $[|\cdot|]_{\textup{U}^{d,p}(G)}$ when $p$ is a positive integer, while Appendix \ref{sec:Udp_not_norms} shows that it can fail otherwise.


\subsection{Notation}

Regard the variable $p$ as a parameter taking values in the set $P$.
For a function $\varphi\colon(0,\infty)\to(0,\infty)$, we write
\[ o_p^{\varepsilon\to0}(\varphi(\varepsilon)) \]
in place of $\psi(\varepsilon,p)$ for an arbitrary function $\psi\colon(0,\infty)\times P \to\mathbb{R}$ that satisfies
\[ \lim_{\varepsilon\to0+} \frac{\psi(\varepsilon,p)}{\varphi(\varepsilon)}=0 \quad \text{for every } p\in P, \]
and
\[ O_p^{\varepsilon\to0}(\varphi(\varepsilon)) \]
in place of $\psi(\varepsilon,p)$ for an arbitrary function $\psi\colon(0,\infty)\times P \to\mathbb{R}$ that satisfies
\[ \limsup_{\varepsilon\to0+} \frac{|\psi(\varepsilon,p)|}{\varphi(\varepsilon)}<\infty\quad \text{for every } p\in P. \]
Completely analogously, one defines the notation
\[ o_p^{n\to\infty}(\varphi(n))\quad \text{and} \quad O_p^{n\to\infty}(\varphi(n)). \]
If $X$ is an arbitrary set, then for two functions $\varphi,\psi \colon X\times P\to [0,\infty)$ we write
\[ \psi(x,p) \lesssim_p \varphi(x,p) \quad\text{or}\quad \varphi(x,p) \gtrsim_p \psi(x,p) \]
if, for every $p\in P$, there exists a constant $C_p\in[0,\infty)$ such that
\[ \psi(x,p) \leq C_p \,\varphi(x,p) \quad \text{for every } x\in X. \]
We write $f_+$ and $f_-$ for the positive and negative parts of a function $f:X\to\R$, defined respectively as
\[ f_+(x) := \max\{f(x),0\} \quad\text{and}\quad f_-(x) := \max\{-f(x),0\} \]
for every $x\in X$.

If $f$ is a complex function on an LCA group $G$, then $\widetilde{f}$ will denote the conjugated and reflected function defined as
\[ \widetilde{f}(x):=\overline{f(-x)} \quad\text{for } x\in G. \]
We normalise the Fourier transform of $f\in\textup{L}^1(\R^n)$ by
\[ \widehat{f}(\xi) := \int_{\R^n} f(x)e^{-2\pi i x\cdot\xi}\dd x \quad\text{for } \xi\in\R^n. \]
By a change of variables, we have $\widehat{\widetilde{f}}(\xi) = \overline{\widehat{f}(\xi)}.$

The Lebesgue measure of a measurable set $E\subseteq\R^n$ will be written as $|E|$. The cardinality of a finite set $A$ will also be written as $|A|$; this should not cause confusion.
The torus $\T=\R/\Z$ is equipped with its normalised Haar measure, which coincides with the Lebesgue measure on $[0,1)$ via the standard choice of the fundamental domain.
We write $\nu^d$ for the product measure
\[ \underbrace{\nu\times\cdots\times\nu}_{d\text{ factors}}. \]

The ``sinc'' function is defined on the real numbers as
\[ \operatorname{sinc}(t)
:= \begin{cases}
\displaystyle\frac{\sin(\pi t)}{\pi t} & \text{for } t\neq0, \\
1 & \text{for } t=0.
\end{cases} \]
Finally, we write $\N$ for the set of strictly positive integers.


\section{Basic properties and degree-lowering inequalities}\label{sec:deglow} For $d\geq2$ and $1\leq i\leq d$, an important formula is the marginal identity
\begin{equation}\label{eq:Rmarg}
\int_G R_d f(h_1,\ldots,h_d)\dd\mu(h_i) =\bigl|R_{d-1}f(h_1,\ldots,h_{i-1},h_{i+1},\ldots,h_d)\bigr|^2
\end{equation}
whenever $f$ is nonnegative or the relevant integrals are absolutely convergent.
Using this, the second identity of the following three is deduced:
\begin{equation}\label{eq:RL12}
\left. \begin{aligned}
& [f]_{\textup{U}^{d,1}(G)} = \|f\|_{\textup{U}^{d}(G)} \ \text{ for nonnegative }f, \\
& [f]_{\textup{U}^{d,2}(G)} = \|f\|_{\textup{U}^{d+1}(G)} \ \text{ for complex }f\text{ such that }[|f|]_{\textup{U}^{d,2}(G)}<\infty, \\
& [f]_{\textup{U}^{d,\infty}(G)} = \|f\|_{\textup{L}^{2^d}(G)} \ \text{ for complex }f\in\textup{L}^{2^d}(G).
\end{aligned} \right\}
\end{equation}
For any fixed $d,p$ and measurable $f\colon G\to[0,\infty)$,
\[ [f]_{\textup{U}^{d,p}(G)}=0 \quad\Longleftrightarrow\quad f=0 \text{ $\mu$-a.e.\@ on $G$}. \]
Later we will use that, whenever $H$ is a compact open subgroup and $0<p<\infty$,
\begin{equation}\label{eq:subgfun}
R_d\1_H=\mu(H)\1_{H^d},\quad
[\1_H]_{\textup{U}^{d,p}(G)} = \mu(H)^{(d+p)/(2^d p)}.
\end{equation}

For $d\geq1$, the following inequalities hold for exponents $p,p_1,p_2\in(0,\infty]$ related by
\[ \frac{1}{p} = \frac{\vartheta}{p_1} + \frac{1-\vartheta}{p_2} \quad \text{for some }\vartheta\in[0,1]. \]
If $f\colon G\to[0,\infty)$ is measurable, then the log-convexity of Lebesgue norms yields
\begin{equation}\label{eq:just_log_conv}
[f]_{\textup{U}^{d,p}(G)} \leq [f]_{\textup{U}^{d,p_1}(G)}^{\vartheta} [f]_{\textup{U}^{d,p_2}(G)}^{1-\vartheta}.
\end{equation}
Moreover, if $f_1,f_2\colon G\to[0,\infty)$ are measurable functions, then H\"{o}lder's inequality first gives
\[ R_d \bigl(f_1^{\vartheta} f_2^{1-\vartheta}\bigr) \leq (R_d f_1)^{\vartheta} (R_d f_2)^{1-\vartheta} \]
and then
\begin{equation}\label{eq:just_Hoelder}
\bigl[f_1^{\vartheta} f_2^{1-\vartheta}\bigr]_{\textup{U}^{d,p}(G)} \leq [f_1]_{\textup{U}^{d,p_1}(G)}^{\vartheta} [f_2]_{\textup{U}^{d,p_2}(G)}^{1-\vartheta}.
\end{equation}

More generally, if $(X,\nu)=(X_1\times\cdots\times X_d,\nu_1\times\cdots\times \nu_d)$ is a product of $\sigma$-finite measure spaces $(X_i,\nu_i)$, we consider
\begin{align*}
X^{(i)} & := X_1\times\cdots\times X_{i-1}\times X_{i+1}\times\cdots\times X_d, \\
\nu^{(i)} & := \nu_1\times\cdots\times \nu_{i-1}\times \nu_{i+1}\times\cdots\times \nu_d
\end{align*}
for $i=1,\ldots,d$.
If functions $F_i\colon X^{(i)}\to[0,\infty)$ are measurable with respect to the completion of the product $\sigma$-algebras, then
\begin{equation}\label{eq:LoomisWhitney} \int_{X}\prod_{i=1}^dF_i(\pi_i x)\dd\nu(x) \le \prod_{i=1}^d \|F_i\|_{\textup{L}^{d-1}(X^{(i)},\nu^{(i)})}, \end{equation}
where $\pi_i\colon X\to X^{(i)}$ denotes the $i$th ``coordinate deletion'' projection.
This is the celebrated Loomis--Whitney inequality~\cite{LoomisWhitney}.

The degree-lowering inequalities of Theorem~\ref{thm:diffUdp}(a) are a simple consequence of the following lemma due to Bennett and Tao \cite{BennettTao}. Their estimates can be considered to be adjoint versions of the (generalised) Loomis--Whitney inequality.

\begin{lemma}[from {\cite{BennettTao}}]
\label{lem:wadjLW}
Let $d\geq2$ and let $a,b_i>0$ and $\theta_i\in \R$ satisfy $\sum_i\theta_i=1$ and
\[ \frac{1}{d-1}\Bigl(1-\frac{1}{a}\Bigr) = \theta_i\Bigl(1-\frac{1}{b_i}\Bigr) \quad\text{for } 1\leq i\leq d. \]
Let $F:X\to[0,\infty)$ be measurable with respect to the completion of the product $\sigma$-algebra, and, for $1\leq i\leq d$, define the $i$th marginal by
\[F_i(x_1,\ldots,x_{i-1},x_{i+1},\ldots,x_d) :=\int_{X_i}F(x_1,\ldots,x_d)\dd \nu_i (x_i).\]
If $0<a\leq1$ and all $\theta_i>0$, then
\[ \|F\|_{\textup{L}^a(X,\nu)} \leq\prod_{i=1}^d \|F_i\|_{\textup{L}^{b_i}(X^{(i)},\nu^{(i)})}^{\theta_i}. \]
If $1<a<\infty$ and precisely one $\theta_i>0$, then
\[ \|F\|_{\textup{L}^a(X,\nu)} \geq\prod_{i=1}^d \|F_i\|_{\textup{L}^{b_i}(X^{(i)},\nu^{(i)})}^{\theta_i}. \]
To make sense of this final inequality, we assume that the norms are positive and finite.
\end{lemma}

The first estimate of the previous lemma is precisely the adjoint Loomis--Whitney inequality of Bennett and Tao \cite[Thm.~2.1(ii)]{BennettTao}.
The reverse estimate is a specialisation of their reverse adjoint Brascamp--Lieb inequality \cite[Thm.~3.1]{BennettTao}. That inequality was formulated on Euclidean spaces in \cite{BennettTao}, but the proof used only the Loomis--Whitney inequality and H\"{o}lder's inequality, so it applies verbatim to the present setting of $\sigma$-finite measure spaces.
The elements of the proof of Lemma~\ref{lem:wadjLW} will still need to be recalled in Section \ref{sec:lowering_extremisers}, where we will study the near-extremisers.

\begin{proof}[Proof of Theorem~\ref{thm:diffUdp}(a)]
Apply Lemma~\ref{lem:wadjLW} with $a=p$ and $b_i=p_i/2$ to the function $F = R_d f$ on $G^d$. By \eqref{eq:Rmarg}, the $i$th marginal of $F$ is precisely $F_i = (R_{d-1}f)^2$, and hence
\begin{align*}
\|F\|_{\textup{L}^{a}(G^d)} & = [f]_{\textup{U}^{d,p}(G)}^{2^d}, \\
\|F_i\|_{\textup{L}^{b_i}(G^{d-1})} & = \|R_{d-1}f\|_{\textup{L}^{p_i}(G^{d-1})}^2 = [f]_{\textup{U}^{d-1,p_i}(G)}^{2^d}.
\end{align*}
Taking the power $2^{-d}$ gives precisely the two desired inequalities.
\end{proof}


\section{Lebesgue estimates and near-extremisers}
\label{sec:Lebest}

We first formulate an estimate obtained by iterative application of symmetric instances of Theorem~\ref{thm:diffUdp}(a).

\begin{lemma}
Let $d\geq2$, let $1\le k\le d$ and let $0<q\leq2$. Then
\begin{equation}\label{eq:fwditer}
[f]_{\textup{U}^{d,\,dq/(2^{d-k}(k+q)-q)}(G)} \leq [f]_{\textup{U}^{k,q}(G)}
\end{equation}
whenever $f\colon G\to[0,\infty)$ is measurable. In particular, taking $k=q=1$,
\begin{equation}\label{eq:lowLpend}
[f]_{\textup{U}^{d,\,d/(2^d-1)}(G)} \leq \|f\|_{\textup{L}^1(G)}.
\end{equation}
\end{lemma}

\begin{proof} When $k=d$, there is nothing to prove. When $k=d-1$,
take all weights $\theta_i$ equal to $1/d$ in Theorem~\ref{thm:diffUdp}(a). If we take every lower-order exponent $p_i$ to be $q$, then the first inequality from that theorem gives
\begin{equation}\label{eq:fwdstep}
[f]_{\textup{U}^{d,\,dq/(2(d-1)+q)}(G)} \leq [f]_{\textup{U}^{d-1,q}(G)}.
\end{equation}
Note that we needed $q\leq2$ to guarantee that $p:=dq/(2(d-1)+q)\leq1$.
When $q=2$, \eqref{eq:fwdstep} is actually the identity
\[ [f]_{\textup{U}^{d,1}(G)}=[f]_{\textup{U}^{d-1,2}(G)}, \]
which follows from \eqref{eq:RL12}.

We now prove \eqref{eq:fwditer} by induction on $d\geq2$. Note that the base case $d=2$ (when we only allow $k=1$) is a consequence of \eqref{eq:fwdstep}.
Take some $d\geq3$ and assume that the claim holds for $d-1$. If $k=d-1$, then the desired estimate is again just \eqref{eq:fwdstep}.
Otherwise, if $1\leq k\leq d-2$, then we use the induction hypothesis to obtain
\begin{equation}\label{eq:fwdstep2}
[f]_{\textup{U}^{d-1,\,(d-1)q/(2^{d-1-k}(k+q)-q)}(G)} \leq [f]_{\textup{U}^{k,q}(G)},
\end{equation}
while \eqref{eq:fwdstep} applied with
\[ \frac{(d-1)q}{2^{d-1-k}(k+q)-q} \in (0,1] \]
in place of $q$ gives
\begin{equation}\label{eq:fwdstep3}
[f]_{\textup{U}^{d,\,dq/(2^{d-k}(k+q)-q)}(G)} \leq [f]_{\textup{U}^{d-1,\,(d-1)q/(2^{d-1-k}(k+q)-q)}(G)}.
\end{equation}
Combining \eqref{eq:fwdstep2} and \eqref{eq:fwdstep3}, we deduce \eqref{eq:fwditer} and the induction step is complete.
\end{proof}

\begin{proof}[Proof of Theorem~\ref{thm:Lpest}(a)]
The lower endpoint $p=d/(2^d-1)$ is precisely \eqref{eq:lowLpend}, so we consider $d/(2^d-1)<p<\infty$.
We can also assume that $\|f\|_{\textup{L}^{2^d p/(d+p)}(G)}$ is positive and finite and so, by homogeneity of the desired estimate, we can normalise as
\[ \|f\|_{\textup{L}^{2^d p/(d+p)}(G)}=1. \]
Splitting $f= f_1^{\vartheta} f_2^{1-\vartheta}$, where
\[ f_1 := f^{2^d p/(d+p)}, \quad f_2 := f^{p/(d+p)}, \quad \vartheta:=\frac{d}{p(2^d-1)}, \]
we have that
\[ \|f_1\|_{\textup{L}^1(G)} = \|f_2\|_{\textup{L}^{2^d}(G)}=1. \]
Using estimate \eqref {eq:just_Hoelder}, which was a simple consequence of H\"{o}lder's inequality, with $p_1=d/(2^d-1)$ and $p_2=\infty$ gives us
\begin{align*}
[f]_{\textup{U}^{d,p}(G)}
& \leq [f_1]_{\textup{U}^{d,\,d/(2^d-1)}(G)}^{\vartheta} [f_2]_{\textup{U}^{d,\infty}(G)}^{1-\vartheta} \\
& \leq \|f_1\|_{\textup{L}^1(G)}^{\vartheta} \|f_2\|_{\textup{L}^{2^d}(G)}^{1-\vartheta} =1,
\end{align*}
where the second inequality follows from \eqref{eq:lowLpend} and \eqref{eq:RL12}. This proves the assertion.

By \eqref{eq:subgfun}, equality holds for $f=c\1_{x_0+H}$ whenever $x_0\in G$ and $H$ is a compact open subgroup of $G$.
\end{proof}

\begin{proof}[Proof of Theorem~\ref{thm:Lpest}(b)]
By homogeneity we may assume that
\[ \|f\|_{\textup{L}^{2^d p/(d+p)}(G)}=1, \quad [f]_{\textup{U}^{d,p}(G)}\geq1-\varepsilon. \]

First suppose that $d/(2^d-1)<p<1$. Choose $0<\vartheta<1$ so that
\[ \frac{1}{p} = \frac{\vartheta(2^d-1)}d + \frac{1-\vartheta}{1},
\quad \text{i.e.,} \quad \vartheta = \frac{d(1-p)}{(2^d-d-1)p}, \]
and put
\[ g := f^{2^d p/(d+p)}, \quad u := f^{p(d+1)/(d+p)}. \]
Then $f=g^\vartheta u^{1-\vartheta}$ and
\begin{align}
\|g\|_{\textup{L}^1(G)} & = 1 \label{eq:gis1} \\
\|u\|_{\textup{L}^{2^d/(d+1)}(G)} & = 1. \label{eq:uis1}
\end{align}
Estimate \eqref {eq:just_Hoelder} with $p_1=d/(2^d-1)$ and $p_2=1$, followed by \eqref{eq:lowLpend} and \eqref{eq:gis1}, now gives
\begin{align*}
1-\varepsilon & \leq [f]_{\textup{U}^{d,p}(G)} \leq [g]_{\textup{U}^{d,d/(2^d-1)}(G)}^{\vartheta} [u]_{\textup{U}^{d,1}(G)}^{1-\vartheta} \\
& \leq \|g\|_{\textup{L}^1(G)}^{\vartheta} \|u\|_{\textup{U}^{d}(G)}^{1-\vartheta} = \|u\|_{\textup{U}^{d}(G)}^{1-\vartheta},
\end{align*}
so that $\|u\|_{\textup{U}^{d}(G)} \geq 1 - O_{d,p}^{\varepsilon\to0}(\varepsilon)$.

If $p\geq1$, define $u$ as before and set
\[ v := f^{p/(d+p)}. \]
Now $f=u^{1/p}v^{1-1/p}$. We again have \eqref{eq:uis1} and
\begin{equation}\label{eq:vis1}
\|v\|_{\textup{L}^{2^d}(G)}=1.
\end{equation}
By estimate \eqref {eq:just_Hoelder} applied with $p_1=1$ and $p_2=\infty$, combined with \eqref{eq:RL12} and \eqref{eq:vis1},
\begin{align*}
1-\varepsilon & \leq [f]_{\textup{U}^{d,p}(G)} \leq [u]_{\textup{U}^{d,1}(G)}^{1/p} [v]_{\textup{U}^{d,\infty}(G)}^{1-1/p} \\
& = \|u\|_{\textup{U}^{d}(G)}^{1/p} \|v\|_{\textup{L}^{2^d}(G)}^{1-1/p} = \|u\|_{\textup{U}^{d}(G)}^{1/p},
\end{align*}
which implies $\|u\|_{\textup{U}^{d}(G)} \geq 1 - O_{p}^{\varepsilon\to0}(\varepsilon)$.

Thus, in all cases of $p$,
\[ \|u\|_{\textup{U}^d(G)} \geq1-O_{d,p}^{\varepsilon\to0}(\varepsilon). \]
In combination with \eqref{eq:uis1} we see that the near-extremiser theorem of Eisner and Tao \cite[Sec.\,6]{EisnerTao} applies to $u$; recall \eqref{eq:ET_near}. There exist a compact open subgroup $H$ of $G$ and a point $x_0\in G$ such that
\[ \Bigl\|u-\frac{\1_{x_0+H}}{\mu(H)^{(d+1)/2^d}}\Bigr\|_{\textup{L}^{2^d/(d+1)}(G)} = o^{\varepsilon\to0}_{d,p}(1). \]
Indeed, the complex-valued statement \eqref{eq:ET_near} allows a polynomial phase, but, because $u$ is nonnegative, taking absolute values removes it without increasing the error.
We temporarily denote $r:=2^d/(d+1)$ and $u_0:=\mu(H)^{-1/r}\1_{x_0+H}$, so that the last display becomes
\[ \|u-u_0\|_{\textup{L}^{r}(G)} = o^{\varepsilon\to0}_{d,p}(1). \]
On the one hand, for every $\gamma\in(0,1]$ we have the pointwise inequality
\[ |u^\gamma-u_0^\gamma|\leq |u-u_0|^\gamma, \]
so that
\[ \|u^\gamma-u_0^\gamma\|_{\textup{L}^{r/\gamma}(G)}
\leq \|u-u_0\|_{\textup{L}^{r}(G)}^\gamma
= o^{\varepsilon\to0}_{d,p,\gamma}(1). \]
On the other hand, for every $\gamma\in(1,r)$ we use
\[ |u^\gamma-u_0^\gamma| \leq |u-u_0| \bigl( \gamma |u|^{\gamma-1} + \gamma |u_0|^{\gamma-1} \bigr), \]
so that H\"{o}lder's inequality yields
\[ \|u^\gamma-u_0^\gamma\|_{\textup{L}^{r/\gamma}(G)}
\leq \|u-u_0\|_{\textup{L}^{r}(G)} \bigl(\underbrace{\gamma \|u\|_{\textup{L}^r(G)}^{\gamma-1} + \gamma \|u_0\|_{\textup{L}^r(G)}^{\gamma-1}}_{2\gamma}\bigr)
= o^{\varepsilon\to0}_{d,p,\gamma}(1). \]
Taking $\gamma=(d+p)/(p(d+1))$ and recalling the definitions of $u$, $u_0$ and $r$, we obtain
\[ \Bigl\|f-\frac{\1_{x_0+H}}{\mu(H)^{(d+p)/(2^d p)}}\Bigr\|_{\textup{L}^{2^d p/(d+p)}(G)} = o^{\varepsilon\to0}_{d,p}(1). \]
Undoing the initial normalisation proves the assertion.
\end{proof}


\section{The sharp Euclidean inequality}
\label{sec:Euclid}

We now compute the best constant on $\R^n$. If $t'$ denotes the conjugate exponent to $t\in(1,\infty)$, put
\[ A_t := \frac{t^{1/(2t)}}{(t')^{1/(2t')}}, \]
with $A_1$ and $A_\infty$ both being defined as $1$.

\begin{proof}[Proof of Theorem~\ref{thm:sharpRn}(a)]
The following proof is a minor modification of the one from \cite[Sec.\,9]{EisnerTao}, only for $[\,\cdot\,]_{\textup{U}^{d,p}}$ rather than $\|\cdot\|_{\textup{U}^d}$.
We can work with an $f$ that is bounded and compactly supported.

For $h'=(h_1,\ldots,h_{d-1})\in(\R^n)^{d-1}$, set
\[ F_{h'}(x) := \prod_{(\omega_1,\ldots,\omega_{d-1})\in\{0,1\}^{d-1}}f(x+\omega_1 h_1+\cdots+\omega_{d-1} h_{d-1}), \]
so that, for yet another vector $h_d\in\R^n$,
\[ (R_d f)(h',h_d)
= \int_{\R^n} F_{h'}(x) F_{h'}(x+h_d)\, \dd x
= \bigl(F_{h'}\ast \widetilde{F}_{h'}\bigr)(h_d). \]
The sharp form of Young's inequality due to Beckner \cite{Beckner1975} gives
\[ \bigl\|F_{h'}\ast \widetilde{F}_{h'}\bigr\|_{\textup{L}^{p}(\R^n)}
\leq \Bigl(\frac{A_{2p/(p+1)}^2}{A_p}\Bigr)^{n} \|F_{h'}\|_{\textup{L}^{2p/(p+1)}(\R^n)}^2 \]
and the identity
\[ \|F_{h'}\|_{\textup{L}^{2p/(p+1)}(\R^n)}^{2p/(p+1)} = R_{d-1}(f^{2p/(p+1)})(h') \]
then yields the recursive estimate
\begin{equation}\label{eq:Youngrec}
[f]_{\textup{U}^{d,p}(\R^n)}
\leq \Bigl(\frac{A_{2p/(p+1)}^2}{A_p}\Bigr)^{n/2^d} \bigl[f^{2p/(p+1)}\bigr]_{\textup{U}^{d-1,p+1}(\R^n)}^{(p+1)/(2p)}
\end{equation}
for $d\geq1$. In order to make this estimate meaningful and true for $d=1$ as well, we use the convention $[g]_{\textup{U}^{0,q}(\R^n)}=\int_{\R^n}g$ for every $q\in(0,\infty)$.

To prove the desired estimate of the theorem statement, we repeatedly apply \eqref{eq:Youngrec}.
Namely, for any $j\in\{0,1,\ldots,d-1\}$ using \eqref{eq:Youngrec} with $d$ replaced by $d-j$, $p$ replaced by $p+j$ and $f$ replaced by $f^{2^j p/(p+j)}$ we obtain
\[ \bigl[f^{2^j p/(p+j)}\bigr]_{\textup{U}^{d-j,p+j}(\R^n)}
\leq \Bigl(\frac{A_{2(p+j)/(p+j+1)}^2}{A_{p+j}}\Bigr)^{n/2^{d-j}} \bigl[f^{2^{j+1} p/(p+j+1)}\bigr]_{\textup{U}^{d-j-1,p+j+1}(\R^n)}^{(p+j+1)/(2(p+j))}. \]
Raising this to $(p+j)/(2^j p)$ we get
\begin{equation}\label{eq:Youngrec2}
\begin{aligned}
\bigl[f^{2^j p/(p+j)}\bigr]_{\textup{U}^{d-j,p+j}(\R^n)}^{(p+j)/(2^j p)} \leq\, \Bigl(\frac{A_{2(p+j)/(p+j+1)}^2}{A_{p+j}}&\Bigr)^{n(p+j)/(2^d p)} \\
 &\bigl[f^{2^{j+1} p/(p+j+1)}\bigr]_{\textup{U}^{d-j-1,p+j+1}(\R^n)}^{(p+j+1)/(2^{j+1}p)} .
\end{aligned}
\end{equation}
Now observe that the inequalities \eqref{eq:Youngrec2} indexed by $j=0,1,\ldots,d-1$ can be chained together to yield
\begin{equation}\label{eq:Youngiter}
[f]_{\textup{U}^{d,p}(\R^n)}
\leq \prod_{j=0}^{d-1} \Bigl(\frac{A_{2(p+j)/(p+j+1)}^2}{A_{p+j}}\Bigr)^{n(p+j)/(2^d p)} \|f\|_{\textup{L}^{2^d p/(d+p)}(\R^n)}.
\end{equation}

It remains to verify that \eqref{eq:Youngiter} is precisely the asserted estimate with the claimed constant.
This follows from the computation
\[ \prod_{j=0}^{d-1} \frac{A_{2(p+j)/(p+j+1)}^{2(p+j)}}{A_{p+j}^{p+j}}
= \prod_{j=0}^{d-1} \frac{2(p+j)^{(p+j)/2}}{(p+j+1)^{(p+j+1)/2}} = 2^d \frac{p^{p/2}}{(p+d)^{(p+d)/2}} \]
and raising the computed product to the power $n/(2^d p)$.

To prove sharpness, let $g(x)=e^{-\pi|x|^2}$, when we can compute
\[ R_d g(h_1,\ldots,h_d) = 2^{-dn/2}\exp\Bigl(-\frac{\pi 2^d}{4} \sum_{i=1}^d|h_i|^2\Bigr) \]
for $h_1,\ldots,h_d\in\R^n$.
It follows that
\begin{align*}
[g]_{\textup{U}^{d,p}(\R^n)}
& = 2^{-dn/(2^{d+1})}(2^{d-2}p)^{-nd/(2^{d+1}p)}, \\
\|g\|_{\textup{L}^{2^d p/(d+p)}(\R^n)}
& = \Bigl(\frac{2^d p}{d+p}\Bigr)^{-n(d+p)/(2^{d+1}p)},
\end{align*}
and their ratio is precisely the displayed constant in the theorem. Moreover, Gaussians attain equality at every sharp Young step \eqref{eq:Youngrec2}.
Conversely, since $d\geq2$, the final Young inequality in \eqref{eq:Youngrec2}, the one for $j=d-1$, is nondegenerate and its equality forces every positive power of every extremiser $f$ to be a Gaussian; see more details in the proof of part (b) below.
\end{proof}

\begin{proof}[Proof of Theorem~\ref{thm:sharpRn}(b)]
We proceed similarly to \cite[Sec.\,3]{Neuman2020}.
By homogeneity, normalise $\|f\|_{\textup{L}^{2^d p/(d+p)}(\R^n)}=1$, so that the hypothesis is
\[ [f]_{\textup{U}^{d,p}(\R^n)}
\geq(1-\varepsilon)\Bigl(\frac{2^{2d}p^p}{(d+p)^{d+p}}\Bigr)^{n/(2^{d+1}p)}. \]
Since the ratio of the left-hand side to the right-hand side of \eqref{eq:Youngiter} is at least $1-\varepsilon$, the same is true at each application of sharp Young's inequality in the derivation of \eqref{eq:Youngiter}, i.e., besides \eqref{eq:Youngrec2} we also have
\begin{align}\label{eq:Youngrec3}
\bigl[f^{2^j p/(p+j)}\bigr]_{\textup{U}^{d-j,p+j}(\R^n)}^{(p+j)/(2^j p)} \geq\, (1-\varepsilon) \Bigl(\frac{A_{2(p+j)/(p+j+1)}^2}{A_{p+j}}&\Bigr)^{n(p+j)/(2^d p)} \\
& \bigl[f^{2^{j+1} p/(p+j+1)}\bigr]_{\textup{U}^{d-j-1,p+j+1}(\R^n)}^{(p+j+1)/(2^{j+1}p)} .\nonumber
\end{align}
In particular, consider the final Young step, obtained for $j=d-1$, and put
\[ u := f^{2^{d-1}p/(p+d-1)}. \]
Then
\[ \|u\|_{\textup{L}^{2(p+d-1)/(p+d)}(\R^n)}=1 \]
and \eqref{eq:Youngrec3} reads
\begin{equation}\label{eq:Yfinal}
\|u\ast \widetilde{u}\|_{\textup{L}^{p+d-1}(\R^n)}
\geq \bigl( 1-O_{d,p}^{\varepsilon\to0}(\varepsilon) \bigr)
\Bigl(\frac{A_{2(p+d-1)/(p+d)}^2}{A_{p+d-1}}\Bigr)^n \|u\|_{\textup{L}^{2(p+d-1)/(p+d)}(\R^n)}^2 .
\end{equation}
Christ's characterisation of near-extremisers for the sharp Young inequality \cite[Thm.\,1.1]{ChristYoung2019}, which applies to \eqref{eq:Yfinal}, gives a Gaussian $u_0$ such that
\[ \|u_0\|_{\textup{L}^{2(p+d-1)/(p+d)}(\R^n)}=1 \]
and
\begin{equation}\label{eq:YGauss}
\|u-u_0\|_{\textup{L}^{2(p+d-1)/(p+d)}(\R^n)} = o^{\varepsilon\to0}_{d,p,n}(1).
\end{equation}
If the Gaussian supplied by the theorem has a phase, replace it by its absolute value. This does not increase the error because $u$ is nonnegative.
Finally put
\[ g := u_0^{(p+d-1)/(2^{d-1}p)}. \]
This is a positive Gaussian and
\[ \|g\|_{\textup{L}^{2^d p/(d+p)}(\R^n)}=1. \]
The same reasoning as at the end of the proof of Theorem~\ref{thm:Lpest}(b), only with $\gamma=(p+d-1)/(2^{d-1}p)$, then transforms \eqref{eq:YGauss} into
\[ \|f-g\|_{\textup{L}^{2^d p/(d+p)}(\R^n)} = o^{\varepsilon\to0}_{d,p,n}(1). \]
Undoing the normalisation proves the theorem.
\end{proof}

We remark that part (b) of the previous proof was rather short due to our assumption that $f$ is nonnegative, which is natural in the context of functionals $[\,\cdot\,]_{\textup{U}^{d,p}}$. Characterising near-extremisers for the true Gowers norms on Euclidean spaces is technically more complicated, precisely due to the fact that $f$ can be complex; see \cite{Neuman2020}.


\section{Near-extremisers of the degree-lowering inequalities}
\label{sec:lowering_extremisers}

The following stability of H\"{o}lder's inequality is a minor variant of \cite[Lemma\,5.1]{EisnerTao}; it also follows from the work of Aldaz \cite[Cor.\,2.4]{AldazAMGM}.

\begin{lemma}[from {\cite{EisnerTao}}]\label{lem:stHold}
Let $d\geq2$, and let $\gamma_1,\ldots,\gamma_d\in(0,1)$ be such that $\sum_i\gamma_i=1$. Write $\boldsymbol{\gamma}=(\gamma_1,\ldots,\gamma_d)$. Let $u_1,\ldots,u_d\colon X\to[0,\infty)$ be measurable functions on a $\sigma$-finite measure space $(X,\nu)$ that satisfy $\int_X u_i\, \dd \nu = 1$ for $i=1,\ldots,d$. If
\[ \int_X\prod_{i=1}^d u_i^{\gamma_i}\dd\nu\geq1-\varepsilon, \]
then, for $1\leq i,j\leq d$,
\begin{align*}
\|u_i-u_j\|_{\textup{L}^1(X)}
& = o^{\varepsilon\to0}_{\boldsymbol{\gamma}}(1),\\
\Bigl\|u_i-\prod_{l=1}^d u_l^{\gamma_l}\Bigr\|_{\textup{L}^1(X)}
& = o^{\varepsilon\to0}_{\boldsymbol{\gamma}}(1).
\end{align*}
\end{lemma}

The next lemma characterises near-extremisers of the Loomis--Whitney inequality. It is folklore and appears as a special case of Cook and Dembo's stability theorem for Finner's inequality \cite[Thm.\,1.17]{CookDembo}. We explain the specialisation and the passage from probability spaces to $\sigma$-finite measure spaces.
Let $\pi_i\colon X^d\to X^{d-1}$ denote the $i$th ``coordinate deletion'' projection.

\begin{lemma}
\label{lem:LWfact}
Let $d\geq2$ and let $(X,\nu)$ be a $\sigma$-finite measure space. Let the measurable functions $F_i\colon X^{d-1}\to[0,\infty)$ satisfy
\[ \int_{X^{d-1}}F_i^{d-1}\dd\nu^{d-1}=1, \]
 for $1\leq i\leq d$, and
\[ \int_{X^d}\prod_{i=1}^dF_i(\pi_i x)\dd\nu^d(x) \geq 1-\varepsilon. \]
Then there exist measurable functions $r_1,\ldots,r_d\colon X\to[0,\infty)$ such that
\[ \int_X r_i\, \dd\nu = 1, \]
for $1\leq i\leq d$, and
\begin{equation}\label{eq:LWfunfact}
\Bigl\|F_i^{d-1}-\bigotimes_{j\ne i}r_j\Bigr\|_{\textup{L}^1(X^{d-1})} = o^{\varepsilon\to0}_d(1) .
\end{equation}
\end{lemma}

\begin{proof}
Choose a measurable function $w>0$ $\nu$-a.e.\@ with $\int_Xw\dd\nu=1$, which is possible by $\sigma$-finiteness, and put $\dd\lambda=w\dd\nu$.
In \cite[Thm.\,1.17]{CookDembo}, take the coordinate set $V=\{1,\ldots,d\}$, the subsets $S_i=V\setminus\{i\}$ with weights $1/(d-1)$ and the probability space $(X,\lambda)$ in every coordinate. Define
\[  g_i(x_{S_i}):=\frac{F_i(x_{S_i})^{d-1}}{\prod_{j\ne i}w(x_j)}. \]
Then $\int g_i\dd\lambda^{d-1}=1$, and each coordinate belongs to exactly $d-1$ of the sets $S_i$, so
\[ \int_{X^d}\prod_{i=1}^d g_i(\pi_i x)^{1/(d-1)}\dd\lambda^d(x)
 =\int_{X^d}\prod_{i=1}^dF_i(\pi_i x)\dd\nu^d(x)\geq1-\varepsilon. \]
The equivalence classes of coordinates in that theorem are singletons. It therefore supplies nonnegative functions $h_j$ with $\int_Xh_j\dd\lambda=1$ such that
\[ \Bigl\|g_i-\bigotimes_{j\ne i}h_j\Bigr\|_{\textup{L}^1(\lambda^{d-1})} \lesssim_d\varepsilon^{c_d} \]
for some $c_d>0$. Set $r_j:=wh_j$. Then $\int_Xr_j\dd\nu=1$, and cancellation of the densities gives
\[ \Bigl\|F_i^{d-1}-\bigotimes_{j\ne i}r_j\Bigr\|_{\textup{L}^1(\nu^{d-1})}
 =\Bigl\|g_i-\bigotimes_{j\ne i}h_j\Bigr\|_{\textup{L}^1(\lambda^{d-1})}, \]
which proves \eqref{eq:LWfunfact}.
\end{proof}

We now prove the tensor product description of the near-extremisers in Theorem~\ref{thm:diffUdp}(b).

\begin{proof}[Proof of Theorem~\ref{thm:diffUdp}(b)]
Suppose that the parameters satisfy the assumptions of Theorem~\ref{thm:diffUdp}(b). Denote the ratio of the ``smaller side'' to the ``larger side'' by $\mathfrak{R}(f)$, i.e.,
\[ \mathfrak{R}(f) := \frac{[f]_{\textup{U}^{d,p}(G)}}{\prod_{i=1}^d[f]_{\textup{U}^{d-1,p_i}(G)}^{\theta_i}} \]
in the case of the first, and
\[ \mathfrak{R}(f) := \frac{\prod_{i=1}^d[f]_{\textup{U}^{d-1,p_i}(G)}^{\theta_i}}{[f]_{\textup{U}^{d,p}(G)}} \]
in the second inequality. Thus the hypothesis in either case is now $\mathfrak{R}(f)\geq1-\varepsilon$. In the course of the proof, we revisit the proof of Lemma~\ref{lem:wadjLW} from \cite{BennettTao}.

For $1\leq i\leq d$, define the normalised Loomis--Whitney factor
\[ \Phi_i := \frac{(R_{d-1}f)^{p_i/(d-1)}}{\|R_{d-1}f\|_{\textup{L}^{p_i}(G^{d-1})}^{p_i/(d-1)}} \quad\text{on }G^{d-1}. \]
Thus $\int_{G^{d-1}}\Phi_i^{d-1}\dd\mu^{d-1}=1$. The marginal identity \eqref{eq:Rmarg} can be rewritten as
\begin{equation}\label{eq:wstmarg}
\int_G R_df(h_1,\ldots,h_d)\dd\mu(h_i) = \bigl((R_{d-1}f)\circ\pi_i\bigr)^2,
\end{equation}
where $\pi_i\colon G^d\to G^{d-1}$ denotes the $i$th coordinate hyperplane projection.
It follows that $(R_{d-1}f)\circ\pi_i>0$ almost everywhere on $\{R_df>0\}$. On this set put
\[ T_i:=(R_df)\bigl((R_{d-1}f)\circ\pi_i\bigr)^{p_i-2}, \]
and define $T_i:=0$ elsewhere. Fubini's theorem and \eqref{eq:wstmarg} give
\[ \int_{G^d}T_i\dd\mu^d = \|R_{d-1}f\|_{\textup{L}^{p_i}(G^{d-1})}^{p_i}. \]
Then the ordinary Loomis--Whitney inequality \eqref{eq:LoomisWhitney} gives
\begin{equation}\label{eq:wstLW}
0<\int_{\{R_df>0\}}\prod_{i=1}^d(\Phi_i\circ\pi_i)\dd\mu^d
\leq\int_{G^d}\prod_{i=1}^d(\Phi_i\circ\pi_i)\dd\mu^d \leq1.
\end{equation}

The parameter relation in Theorem~\ref{thm:diffUdp} is equivalent to
\[ \theta_i p(p_i-2)+\frac{(1-p)p_i}{d-1}=0. \]
Consequently, on $\{R_df>0\}$,
\begin{equation}\label{eq:wstptfact}
(R_df)^p = \Bigl(\prod_{i=1}^dT_i^{\theta_i p}\Bigr) \Bigl(\prod_{i=1}^d \bigl((R_{d-1}f)\circ\pi_i\bigr)^{p_i/(d-1)}\Bigr)^{1-p}.
\end{equation}
The same parameter relation also gives
\begin{equation}\label{eq:wstIexp}
\frac{2\theta_i p}{p_i} = \theta_i p+\frac{1-p}{d-1}.
\end{equation}
The functions
\[ w := \frac{(R_df)^p}{\|R_df\|_{\textup{L}^p(G^d)}^p},\quad u_i:=\frac{T_i}{\|R_{d-1}f\|_{\textup{L}^{p_i}(G^{d-1})}^{p_i}}, \]
and
\[ v := \frac{\1_{\{R_df>0\}}\prod_{i=1}^d(\Phi_i\circ\pi_i)}{\int_{\{R_df>0\}}\prod_{i=1}^d(\Phi_i\circ\pi_i)\dd\mu^d} \]
are nonnegative functions on $G^d$ with integral $1$.

Suppose first that $0<p<1$, and define
\[ \eta := \mathfrak{R}(f)^{2^d p}\Bigl( \int_{\{R_df>0\}}\prod_{i=1}^d(\Phi_i\circ\pi_i)\dd\mu^d \Bigr)^{p-1}. \]
Equations \eqref{eq:wstptfact} and \eqref{eq:wstIexp} give
\[ \prod_{i=1}^du_i^{\theta_i p}v^{1-p} = \eta w. \]
The exponents on the left are positive and sum to $1$, so H\"{o}lder's inequality gives $\eta\leq1$ and we have
\[ \mathfrak{R}(f)^{2^d p} = \eta\Bigl(\int_{\{R_df>0\}}\prod_{i=1}^d(\Phi_i\circ\pi_i)\dd\mu^d\Bigr)^{1-p}. \]
Both factors on the right belong to $[0,1]$. Hence $\mathfrak{R}(f)\geq1-\varepsilon$ implies that $\eta$ and the integral in parentheses are both $1-o^{\varepsilon\to0}_{d,p,\mathbf p,\boldsymbol{\theta}}(1)$. Lemma~\ref{lem:stHold} then shows that the $u_i$ and $v$ are pairwise $o^{\varepsilon\to0}_{d,p,\mathbf p,\boldsymbol{\theta}}(1)$ close in $\textup{L}^1(G^d)$. The second conclusion of the same lemma also applies and gives
\begin{equation}\label{eq:wcomdens}
\|w-v\|_{\textup{L}^1(G^d)} + \max_{1\leq i\leq d}\|u_i-v\|_{\textup{L}^1(G^d)} = o^{\varepsilon\to0}_{d,p,\mathbf p,\boldsymbol{\theta}}(1).
\end{equation}

Now suppose that $p>1$, and define instead
\[ \eta := \mathfrak{R}(f)^{2^d p}\Bigl( \int_{\{R_df>0\}}\prod_{i=1}^d(\Phi_i\circ\pi_i)\dd\mu^d \Bigr)^{1-p}. \]
Since $\theta_1 p>1$ and $p-1-p\sum_{i=2}^d\theta_i=\theta_1 p-1$, with $w,u_i,v$ as above, \eqref{eq:wstptfact} is equivalent to
\[ w^{1/(\theta_1 p)} \prod_{i=2}^d u_i^{-\theta_i/\theta_1} v^{(p-1)/(\theta_1 p)} = \eta^{1/(\theta_1 p)} u_1. \]
The exponents on the left are positive and sum to $1$. The integral of the left-hand side is $\eta^{1/(\theta_1 p)}$. H\"{o}lder's inequality gives $\eta\leq1$, and we have the factorisation
\[ \mathfrak{R}(f)^{2^d p} = \eta\Bigl( \int_{\{R_df>0\}}\prod_{i=1}^d(\Phi_i\circ\pi_i)\dd\mu^d \Bigr)^{p-1}. \]
Thus $\eta$ and the integral in parentheses are again $1-o^{\varepsilon\to0}_{d,p,\mathbf p,\boldsymbol{\theta}}(1)$. Lemma~\ref{lem:stHold} first makes $w,u_2,\ldots,u_d,v$ pairwise $o^{\varepsilon\to0}_{d,p,\mathbf p,\boldsymbol{\theta}}(1)$-close in $\textup{L}^1(G^d)$. Its second conclusion applies to the left-hand side of the preceding normalised identity, which equals $\eta^{1/(\theta_1 p)}u_1$. Since $\eta^{1/(\theta_1 p)}=1-o^{\varepsilon\to0}_{d,p,\mathbf p, \boldsymbol{\theta}}(1)$, it gives the same conclusion for $u_1$. We have therefore proved \eqref{eq:wcomdens} again.
In both cases we also have
\begin{equation}\label{eq:wstnear}
\int_{\{R_df>0\}}\prod_{i=1}^d(\Phi_i\circ\pi_i)\dd\mu^d = 1-o^{\varepsilon\to0}_{d,p,\mathbf p,\boldsymbol{\theta}}(1).
\end{equation}

On the copy of $G^{d-1}$ obtained by deleting coordinate $i$, use the factor $\Phi_i$ defined above. By \eqref{eq:wstLW} and \eqref{eq:wstnear},
\[ \int_{G^d}\prod_{i=1}^d (\Phi_i\circ\pi_i)\, \dd\mu^d = 1-o^{\varepsilon\to0}_{d,p,\mathbf p,\boldsymbol{\theta}}(1). \]
Lemma~\ref{lem:LWfact} therefore supplies $r_1,\ldots,r_d\colon G\to[0,\infty)$, which integrate to $1$, such that
\begin{equation}\label{eq:wHtens}
\Bigl\|\frac{(R_{d-1}f)^{p_i}}{\|R_{d-1}f\|_{\textup{L}^{p_i}(G^{d-1})}^{p_i}}-\bigotimes_{j\ne i}r_j\Bigr\|_{\textup{L}^1(G^{d-1})} = o^{\varepsilon\to0}_{d,p,\mathbf p,\boldsymbol{\theta}}(1) \quad\text{for } 1\leq i\leq d.
\end{equation}
This is the normalised form of the second conclusion in Theorem~\ref{thm:diffUdp}(b).

We argue that the same factors also approximate $(R_d f)^p$. The elementary inequality $|s^{1/(d-1)}-t^{1/(d-1)}|^{d-1}\leq|s-t|$ for $s,t\geq0$, together with \eqref{eq:wHtens}, gives
\[ \Bigl\|\Phi_i - \Bigl(\bigotimes_{j\ne i} r_j\Bigr)^{1/(d-1)} \Bigr\|_{\textup{L}^{d-1}(G^{d-1})} = o^{\varepsilon\to0}_{d,p,\mathbf p,\boldsymbol{\theta}}(1). \]
Telescoping the products and applying the ordinary Loomis--Whitney inequality \eqref{eq:LoomisWhitney} to every mixed term now gives
\begin{equation}\label{eq:wfullprod}
\Bigl\| \prod_{i=1}^d(\Phi_i\circ\pi_i) - \bigotimes_{j=1}^d r_j \Bigr\|_{\textup{L}^1(G^d)} = o^{\varepsilon\to0}_{d,p,\mathbf p,\boldsymbol{\theta}}(1).
\end{equation}
All factors appearing without a difference have $\textup{L}^{d-1}(G^{d-1})$-norm one.
Finally, \eqref{eq:wstLW} gives
\[ \Bigl\|v-\prod_{i=1}^d(\Phi_i\circ\pi_i)\Bigr\|_{\textup{L}^1(G^d)}
\leq 2\Bigl(1-\int_{\{R_df>0\}}\prod_{i=1}^d(\Phi_i\circ\pi_i)\dd\mu^d\Bigr)
= o^{\varepsilon\to0}_{d,p,\mathbf p,\boldsymbol{\theta}}(1). \]
Combining this estimate with \eqref{eq:wcomdens} and \eqref{eq:wfullprod} yields
\[ \Bigl\| \frac{(R_df)^p}{\|R_df\|_{\textup{L}^p(G^d)}^p} - \bigotimes_{j=1}^d r_j \Bigr\|_{\textup{L}^1(G^d)}
= o^{\varepsilon\to0}_{d,p,\mathbf p,\boldsymbol{\theta}}(1), \]
which is the normalised form of the first conclusion in Theorem~\ref{thm:diffUdp}(b).
\end{proof}


\section{Improved log-convexity of Gowers norms}
\label{sec:log-convexity}

We will require the following characterisation of the near-extremisers of the geometric Loomis--Whitney inequality.
It follows from Lemma~\ref{lem:LWfact} and the slicing argument below; see also \cite[Cor.\,2]{EFKY16} for a Euclidean version and \cite[Cor.\,3]{EFKY16} for a discrete version.

\begin{lemma}\label{lem:cylLWst}
Let $d\geq2$, let $(X,\nu)$ be a $\sigma$-finite measure space and let $E_i\subseteq X^{d-1}$ have finite product measure $\nu^{d-1}(E_i)>0$. Define
\[ \Omega:=\bigcap_{i=1}^d\pi_i^{-1}(E_i). \]
If, for small $\varepsilon>0$,
\[ \nu^d(\Omega) \geq (1-\varepsilon) \Bigl(\prod_{i=1}^d\nu^{d-1}(E_i)\Bigr)^{1/(d-1)}, \]
 then there are sets $A_i\subseteq X$, of finite measure, such that
\[ \nu^d\bigl(\Omega\mathbin{\triangle}(A_1\times\cdots\times A_d)\bigr) = o^{\varepsilon\to0}_d(1)\,\nu^d(\Omega). \]
\end{lemma}

\begin{proof}
We
apply Lemma~\ref{lem:LWfact} to $F_i:=\1_{E_i}/\nu^{d-1}(E_i)^{1/(d-1)}$.
Taking the $(d-1)$th roots in \eqref{eq:LWfunfact}, telescoping the products and applying the Loomis--Whitney inequality \eqref{eq:LoomisWhitney} to each term gives
\[  \Bigl\|\frac{\1_\Omega}{M}-\bigotimes_{i=1}^d r_i\Bigr\|_{\textup{L}^1(X^d)}=o_d^{\varepsilon\to0}(1), \]
where $M:=(\prod_i\nu^{d-1}(E_i))^{1/(d-1)}$.
Here we also used $|s^{1/(d-1)}-t^{1/(d-1)}|^{d-1}\leq|s-t|$ for $s,t\geq0$.
Since $1-\varepsilon\leq \nu^d(\Omega)/M\leq1$, the triangle inequality yields
\[  \delta:=\Bigl\|\frac{\1_\Omega}{\nu^d(\Omega)}-\bigotimes_{i=1}^d r_i\Bigr\|_{\textup{L}^1(X^d)}=o_d^{\varepsilon\to0}(1). \]
For each $i$, divide the fibrewise error by $\prod_{j\ne i}r_j$, whenever this is positive, and average against the probability measure $(\bigotimes_{j\ne i}r_j)\dd\nu^{d-1}$. Then some measurable fibre $A_i$ of $\Omega$ and some constant $c_i>0$ must satisfy
\[ \|r_i-c_i\1_{A_i}\|_{\textup{L}^1(X)}\leq2\delta. \]
For small $\delta$, this implies $0<\nu(A_i)<\infty$ and $|c_i\nu(A_i)-1|\leq2\delta$, whence
\[ \Bigl\|r_i-\frac{\1_{A_i}}{\nu(A_i)}\Bigr\|_{\textup{L}^1(X)}\leq4\delta. \]
Writing $B:=A_1\times\cdots\times A_d$, by another application of the triangle inequality, this yields
\[ \Bigl\|\frac{\1_\Omega}{\nu^d(\Omega)}-\frac{\1_B}{\nu^d(B)}\Bigr\|_{\textup{L}^1(X^d)}\leq(4d+1)\delta. \]
Finally, the norm on the left-hand side is equal to $2(1-\nu^d(\Omega\cap B)/\max\{\nu^d(\Omega),\nu^d(B)\})$, so that $\nu^d(\Omega\mathbin{\triangle}B)\lesssim_d\delta \nu^d(\Omega)$, as required.
\end{proof}

When $d=2$, $\Omega$ is already a Cartesian product, so this lemma is not needed to answer the original question of Bennett and Tao.

More fundamentally we will need the following result which informally says that if a nonnegative function has three of its Lebesgue (quasi)norms almost equal to~$1$, then it must be close to an indicator function of a set.
Two Lebesgue norms are not enough to obtain this conclusion on $\R^n$, since by scaling a generic function (adjusting its ``height'' and ``width''), one can make both Lebesgue norms exactly~$1$.

\begin{lemma}\label{lem:momrig}
Let $0<a<b<c<\infty$ and $(X,\nu)$ be a $\sigma$-finite measure space. For $\varepsilon>0$, consider a measurable function $F\colon X\to[0,\infty)$ such that
\begin{equation}\label{eq:threemom}
\max \biggl\{ \Big|\int_X F^a\dd\nu-1\Big|, \Big|\int_X F^b\dd\nu-1\Big|, \Big|\int_X F^c\dd\nu-1\Big| \biggr\} \leq\varepsilon.
\end{equation}
Then there is a measurable set $E\subseteq X$ such that
\begin{equation}\label{eq:momz}
\nu(E)=1+o^{\varepsilon\to0}_{a,b,c}(1)
\end{equation}
and, for any $t\in[a,c]$,
\begin{equation}\label{eq:momind}
\int_X|F-\1_E|^t\dd\nu = o^{\varepsilon\to0}_{a,b,c,t}(1).
\end{equation}
\end{lemma}

This lemma follows from Aldaz's stability theorem for interpolation \cite[Thm.\,3.4]{AldazHolder}. We give a short direct proof based only on the stability of H\"{o}lder's inequality from the previous section.

\begin{proof}
We may assume that $0<\varepsilon<1/2$. Let $\vartheta=(c-b)/(c-a)$ so that $b=\vartheta a+(1-\vartheta)c$. Define
\[ u:=\frac{F^a}{\int_XF^a\dd\nu}, \quad v:=\frac{F^c}{\int_XF^c\dd\nu}. \]
The functions $u,v$ are nonnegative and satisfy
\[ \int_X u \,\dd\nu = \int_X v \,\dd\nu = 1, \]
and, by \eqref{eq:threemom}, we have
\[ \int_X u^{\vartheta} v^{1-\vartheta} \dd\nu
= \frac{\int_X F^b \dd\nu}{(\int_X F^a \dd\nu)^{\vartheta}(\int_X F^c \dd\nu)^{1-\vartheta}}
= 1-O_{a,b,c}^{\varepsilon\to0}(\varepsilon). \]
Lemma~\ref{lem:stHold} therefore yields
\[ \int_X |u-v|\, \dd\nu = o^{\varepsilon\to0}_{a,b,c}(1) \]
and, as
\begin{equation}\label{eq:mom0}
\int_X F^a \dd\nu = 1+O_{a,b,c}^{\varepsilon\to0}(\varepsilon), \quad \int_X F^c \dd\nu = 1+O_{a,b,c}^{\varepsilon\to0}(\varepsilon),
\end{equation}
we also have
\begin{equation}\label{eq:momends}
\delta := \int_X |F^a-F^c|\, \dd\nu = o^{\varepsilon\to0}_{a,b,c}(1).
\end{equation}
At this point the $\textup{L}^b$ norm has served its purpose.

Now we can assume that $\varepsilon>0$ is sufficiently small so that $0<\delta<1/2$ (note that we are done if $\delta=0$). Define
\[ E := \Big\{x\in X : F(x)>\frac{1}{2} \Big\} \]
and its subset
\[ E' := \{x\in E : |F(x)-1| < \delta \}. \]
We estimate the error separately on $E'$, $X\setminus E$ and $E\setminus E'$.
Essentially by definition we have that
\begin{equation}\label{eq:mom_aux2}
\int_{E'} |F-\1_E|^t \dd\nu \leq \delta^t\nu(E) \le \delta^t2^a\int_{E} F^a\dd\nu \lesssim_{a} \delta^t.
\end{equation}
At the other extreme, for $x\in X\setminus E$, we have
\[ F(x)^t \leq F(x)^a \lesssim_{a,c} F(x)^a - F(x)^c, \]
so that, by integrating and recalling the definition of $\delta$ in \eqref{eq:momends},
\begin{equation}\label{eq:mom_aux1}
\int_{X\setminus E} |F-\1_E|^t \dd\nu = \int_{X\setminus E} F^t \dd\nu \lesssim_{a,c,t} \delta.
\end{equation}
Finally, for $x\in E\setminus E'$, we have
\[ |F(x)-1|^t \lesssim_{a,c,t} \delta^{\min\{t-1,0\}} |F(x)^a-F(x)^c|. \]
For $1/2<F(x)\leq2$, this follows from $|F(x)^a-F(x)^c|\asymp_{a,c}|F(x)-1|$ and $|F(x)-1|\geq\delta$. For $F(x)>2$, it follows more easily using that $t\leq c$. Integrating yields
\begin{equation}\label{eq:mom_aux3}
\int_{E\setminus E'} |F-\1_E|^t \dd\nu \lesssim_{a,c,t} \delta^{\min\{t,1\}}.
\end{equation}
Then \eqref{eq:momind} follows by adding \eqref{eq:mom_aux2}, \eqref{eq:mom_aux1} and \eqref{eq:mom_aux3} together.

Finally, for $x\in E$, we also have
\[ |F(x)^c-1| \lesssim_{a,c} |F(x)^a-F(x)^c|, \]
so that
\[ \Big| \int_{E} F^c \dd\nu - \nu(E) \Big| \lesssim_{a,c} \delta, \]
which, combined with \eqref{eq:mom0} and the $t=c$ case of \eqref{eq:mom_aux1}, yields \eqref{eq:momz}.
\end{proof}

We will also need the following ``diagonal'' estimate for the autocorrelation $R_2 f$, which can be viewed as a refinement of the $d=2$ case of \eqref{eq:BTlogcv}.

\begin{lemma}\label{lem:diagleak}
Let $f:G\to[0,\infty)$ and $A\subseteq G$ be measurable. Then
\begin{align*}
\iint\limits_{\{(h_1,h_2)\in G^2:h_1+h_2\in A\}}\!\!\!\!R_2 f(h_1,h_2)\dd\mu(h_1)\dd\mu(h_2)
 \leq \|R_1 f\|_{\textup{L}^1(A)}^{1/2} \|f\|_{\textup{L}^1(G)} \|f\|_{\textup{U}^3(G)}^2.
\end{align*}
\end{lemma}

\begin{proof} First we prove that
\begin{align}
\iint\limits_{\{(h_1,h_2)\in G^2:h_1+h_2\in A\}}\!\!\!\!R_2 f(h_1,h_2)\dd\mu(h_1)\dd\mu(h_2)
 \leq \bigl\| f\ast\widetilde{f} \bigr\|_{\textup{L}^1(A)}^{1/2} \|f\ast f\|_{\textup{L}^3(G)}^{3/2} \label{eq:diagleak00}
\end{align}
The integral on the left-hand side unfolds as
\begin{align*}
 \int_A \int_G &R_2f(h,s-h) \dd\mu(h) \dd\mu(s) \\
& = \int_{G^3} f(x) f(x+h) f(x+s-h) f(x+s) \1_A(s)\, \dd\mu(x) \dd\mu(h) \dd\mu(s)
\end{align*}
and the substitution $y=x+s$ transforms it into
\begin{align*}
 \int_{G^3} f(x) &f(x+h) f(y-h) f(y) \1_A(y-x)\, \dd\mu(y) \dd\mu(x) \dd\mu(h) \\
& = \int_{G^2} f(x) f(y) (f\ast f)(x+y) \1_A(y-x)\, \dd\mu(x) \dd\mu(y).
\end{align*}
Using the Cauchy--Schwarz inequality, substituting back $s=y-x$ in the first factor and introducing $t=x+y$ in the second one, this is then bounded by
\begin{align*}
& \leq \Bigl( \int_{G^2} f(x) f(y) \1_A(y-x)\, \dd\mu(x) \dd\mu(y) \Bigr)^{1/2} \\
& \quad\ \Bigl( \int_{G^2} f(x) f(y) (f\ast f)(x+y)^2 \dd\mu(x)\, \dd\mu(y) \Bigr)^{1/2} \\
& = \Bigl( \int_A \int_G f(x) f(x+s)\, \dd\mu(x) \dd\mu(s) \Bigr)^{1/2} \\
& \quad\ \biggl( \int_G (f\ast f)(t)^2 \Bigl( \int_G f(x) f(t-x)\, \dd\mu(x) \Bigr)\, \dd\mu(t) \biggr)^{1/2} \\
& = \bigl\| f\ast\widetilde{f} \bigr\|_{\textup{L}^1(A)}^{1/2} \|f\ast f\|_{\textup{L}^3(G)}^{3/2},
\end{align*}
which completes the proof of \eqref{eq:diagleak00}.

Next we prove that
\begin{align}
\iint\limits_{\{(h_1,h_2)\in G^2:h_1+h_2\in A\}}\!\!\!\!R_2 f(h_1,h_2)\dd\mu(h_1)\dd\mu(h_2)
 \leq \|R_1 f\|_{\textup{L}^1(A)}^{1/2} [f]_{\textup{U}^{1,3}(G)}^3 \label{eq:diagleak0}
\end{align}
Recalling that $R_1 f = f\ast\widetilde{f}$ and the definition of $[\,\cdot\,]_{\textup{U}^{1,3}(G)}$, this will follow from \eqref{eq:diagleak00} once we have established that
\begin{equation}\label{eq:phaserem}
\int_G (f\ast f)^3 \dd\mu \leq \int_G \bigl(f\ast\widetilde{f}\,\bigr)^3\dd\mu.
\end{equation}
Now, the left-hand side of \eqref{eq:phaserem} can be expanded as
\[ \int_{G^4} f(s-x) f(x) f(s-y) f(y) f(s-t) f(t)\, \dd\mu(x) \dd\mu(y) \dd\mu(t) \dd\mu(s), \]
which, after substitutions $x'=s-x$, $u=y-x$, $v=t-x$, becomes
\begin{align*}
\int_{G^2} & \Bigl( \int_G f(x) f(x+u) f(x+v)\, \dd\mu(x) \Bigr) \\
& \Bigl( \int_G f(x') f(x'-u) f(x'-v)\, \dd\mu(x') \Bigr) \dd\mu(u) \dd\mu(v).
\end{align*}
Yet another application of the Cauchy--Schwarz inequality bounds this by
\begin{align*}
 \le\biggl( \int_{G^2} &\Bigl( \int_G f(x) f(x+u) f(x+v)\, \dd\mu(x) \Bigr)^2 \dd\mu(u) \dd\mu(v) \biggr)^{1/2} \\
& \biggl( \int_{G^2} \Bigl( \int_G f(x') f(x'-u) f(x'-v)\, \dd\mu(x') \Bigr)^2 \dd\mu(u) \dd\mu(v) \biggr)^{1/2},
\end{align*}
which easily simplifies to yield \eqref{eq:phaserem} and thus also \eqref{eq:diagleak0}.

Finally, the reverse inequality in Theorem~\ref{thm:diffUdp}(a), with
\[ d=2,\quad p=2,\quad \theta_1=3/2, \quad \theta_2=-1/2,\quad p_1=3,\quad p_2=1, \]
gives
\[ [f]_{\textup{U}^{2,2}(G)} \geq [f]_{\textup{U}^{1,3}(G)}^{3/2} [f]_{\textup{U}^{1,1}(G)}^{-1/2}. \]
Rearranging and squaring, we find
\[ [f]_{\textup{U}^{1,3}(G)}^3 \leq \|f\|_{\textup{U}^1(G)}\|f\|_{\textup{U}^3(G)}^2
= \|f\|_{\textup{L}^1(G)} \|f\|_{\textup{U}^3(G)}^2, \]
so combining this with \eqref{eq:diagleak0} yields the desired inequality.
\end{proof}
The final ingredient is Fournier's improvement of Young's convolution inequality \cite[Thm.\,1]{Fournier}. We only use the particular case $\textup{L}^{4/3}(G)\ast \textup{L}^{4/3}(G)\to\textup{L}^2(G)$. In fact, any triple of indices in the interior of Young's exponent range would suffice for our purpose.

\begin{lemma}[from {\cite{Fournier}}]\label{lem:unifYgap}
There exists an absolute constant $\eta_{\mathrm{Young}}>0$ such that
\[ \|u\ast v\|_{\textup{L}^2(G)}
\leq(1-\eta_{\mathrm{Young}}) \|u\|_{\textup{L}^{4/3}(G)} \|v\|_{\textup{L}^{4/3}(G)} \]
whenever $u,v\in\textup{L}^{4/3}(G)$ and $G$ is an LCA group with no compact open subgroup. The constant is independent of $G$ and of the normalisation of its Haar measure.
\end{lemma}

This improved Young's inequality has other equivalent and readily applicable reformulations; see for example the inverse sumset estimate from \cite[Lemma\,5.5]{EisnerTao}.

\begin{proof}[Proof of Theorem~\ref{thm:mainlog}]
If either $\|f\|_{\textup{U}^{d-1}(G)}$ or $\|f\|_{\textup{U}^{d+1}(G)}$ vanishes, then $f=0$ almost everywhere and the result is immediate, so we assume that this is not the case.
We also assume that $f$ is bounded and compactly supported, since the general case follows by a passage to the limit using the monotone convergence theorem.

We use two specialisations of Theorem~\ref{thm:diffUdp}(a). The first was used previously in Section~\ref{sec:Lebest};
\[ p=\frac{d}{2d-1},\quad \theta_i=\frac{1}{d},\quad p_i=1 \quad\text{for }1\leq i\leq d, \]
which gives
\[ [f]_{\textup{U}^{d,\,d/(2d-1)}(G)}
\leq \|f\|_{\textup{U}^{d-1}(G)}. \]
 On the other hand, the log-convexity property \eqref{eq:just_log_conv} with
\[ p=1, \quad p_1=\frac{d}{2d-1}, \quad p_2=2, \quad \vartheta=\theta(d):=\frac{d}{3d-2} \]
yields
\[ [f]_{\textup{U}^{d,1}(G)} \leq [f]_{\textup{U}^{d,\,d/(2d-1)}(G)}^{\theta(d)} [f]_{\textup{U}^{d,2}(G)}^{1-\theta(d)}. \]
Combining these two estimates and recalling \eqref{eq:RL12}, we obtain a chain of inequalities
\begin{equation}\label{eq:first_chain}
\begin{aligned}
\|f\|_{\textup{U}^{d}(G)} & \leq [f]_{\textup{U}^{d,\,d/(2d-1)}(G)}^{\theta(d)} [f]_{\textup{U}^{d,2}(G)}^{1-\theta(d)} \\
& \leq \|f\|_{\textup{U}^{d-1}(G)}^{\theta(d)} \|f\|_{\textup{U}^{d+1}(G)}^{1-\theta(d)}.
\end{aligned}
\end{equation}
The second specialisation of Theorem~\ref{thm:diffUdp}(a) is given by the choice of parameters
\[ p=2,\quad \theta_1=\frac32,\quad p_1=\frac{6(d-1)}{3d-4},\quad \theta_i=-\frac{1}{2(d-1)},\quad p_i=1 \quad\text{for }2\leq i\leq d, \]
which gives
\[ [f]_{\textup{U}^{d,2}(G)} \geq [f]_{\textup{U}^{d-1,\,6(d-1)/(3d-4)}(G)}^{3/2} [f]_{\textup{U}^{d-1,1}(G)}^{-1/2}. \]
Rearranging and recalling \eqref{eq:RL12}, we find
\[ [f]_{\textup{U}^{d-1,\,6(d-1)/(3d-4)}(G)}
\leq \|f\|_{\textup{U}^{d-1}(G)}^{1/3} \|f\|_{\textup{U}^{d+1}(G)}^{2/3}. \]
On the other hand, the log-convexity \eqref{eq:just_log_conv} can be applied with
\[ p=2, \quad p_1=1, \quad p_2=\frac{6(d-1)}{3d-4}, \quad \vartheta=\frac{1}{3d-2} \]
so that
\[ [f]_{\textup{U}^{d-1,2}(G)} \leq [f]_{\textup{U}^{d-1,1}(G)}^{1/(3d-2)} [f]_{\textup{U}^{d-1,\,6(d-1)/(3d-4)}(G)}^{3(d-1)/(3d-2)}. \]
Thus a second possible inequality chain, proving the same Gowers norm log-convexity estimate, reads
\begin{equation}\label{eq:second_chain}
\begin{aligned}
\|f\|_{\textup{U}^{d}(G)} & \leq \|f\|_{\textup{U}^{d-1}(G)}^{1/(3d-2)} [f]_{\textup{U}^{d-1,\,6(d-1)/(3d-4)}(G)}^{3(d-1)/(3d-2)} \\
& \leq \|f\|_{\textup{U}^{d-1}(G)}^{\theta(d)} \|f\|_{\textup{U}^{d+1}(G)}^{1-\theta(d)}.
\end{aligned}
\end{equation}

It remains to prove that the log-convexity constant can be lowered from $1$. We can scale the function $f$ and the Haar measure $\mu$ (as in \cite[Sec.\,2]{BennettTao}) so that
\begin{equation}\label{eq:mainnorm}
\|f\|_{\textup{U}^{d-1}(G)} = \|f\|_{\textup{U}^{d+1}(G)}=1.
\end{equation}
Supposing that $\|f\|_{\textup{U}^{d}(G)}\geq1-\varepsilon,$
the two chains of inequalities, \eqref{eq:first_chain} and \eqref{eq:second_chain}, imply
\[ \begin{aligned}
[f]_{\textup{U}^{d,\,d/(2d-1)}(G)} & = 1-o^{\varepsilon\to0}_d(1),\\
[f]_{\textup{U}^{d-1,\,6(d-1)/(3d-4)}(G)} & =1-o^{\varepsilon\to0}_d(1),\\
\|f\|_{\textup{U}^{d}(G)} & =1-o^{\varepsilon\to0}_d(1).
\end{aligned} \]
Again using \eqref{eq:RL12}, these estimates combined with \eqref{eq:mainnorm} can be expanded as
\begin{align*}
\int_{G^d}(R_df)^{d/(2d-1)}\dd\mu^d
& = 1-o^{\varepsilon\to0}_d(1),\\
\int_{G^d}R_df\dd\mu^d & = 1-o^{\varepsilon\to0}_d(1),\\
\int_{G^d}(R_df)^2\dd\mu^d & =1,\\
\int_{G^{d-1}}R_{d-1}f\dd\mu^{d-1} & =1,\\
\int_{G^{d-1}}(R_{d-1}f)^2\dd\mu^{d-1} & = 1-o^{\varepsilon\to0}_d(1), \\
\int_{G^{d-1}}(R_{d-1}f)^{6(d-1)/(3d-4)}\dd\mu^{d-1} & = 1-o^{\varepsilon\to0}_d(1).
\end{align*}
Then Lemma~\ref{lem:momrig}, with exponents $d/(2d-1)$, $1$, $2$ on the one hand and exponents $1$, $2$, $6(d-1)/(3d-4)$ on the other, yields measurable sets $S\subseteq G^d$ and $E\subseteq G^{d-1}$ such that
\begin{equation}\label{eq:mainFbool}
\left. \begin{aligned}
\mu^d(S) & =1-o^{\varepsilon\to0}_d(1), \\
\|R_df-\1_S\|_{\textup{L}^1(G^d)} + \|R_df-\1_S\|^2_{\textup{L}^2(G^d)} & =o^{\varepsilon\to0}_d(1), \\
\|(R_df)^{d/(2d-1)}-\1_S\|_{\textup{L}^1(G^d)} & =o^{\varepsilon\to0}_d(1), \\
 \mu^{d-1}(E)&=1-o^{\varepsilon\to0}_d(1), \\
\|R_{d-1}f-\1_E\|_{\textup{L}^1(G^{d-1})} + \|R_{d-1}f-\1_E\|^2_{\textup{L}^2(G^{d-1})} & =o^{\varepsilon\to0}_d(1).
\end{aligned} \right\}
\end{equation}
The estimate for $(R_df)^{d/(2d-1)}$ follows from the $\textup{L}^{d/(2d-1)}$ conclusion of that lemma combined with the elementary inequality $|s^{d/(2d-1)}-t^{d/(2d-1)}|\leq|s-t|^{d/(2d-1)}$ for $s,t\geq0$.

Defining
\[ \Omega := \bigcap_{i=1}^d \pi_i^{-1}(E), \]
we note that $\Omega^c$ can be covered by $d$ copies of $E^c\times G$ oriented in each coordinate direction. Then, using the marginal identity \eqref{eq:Rmarg} and \eqref{eq:mainFbool}, we find
\[ \int_{\Omega^c}R_df\dd\mu^d \leq d\int_{E^c} (R_{d-1}f)^2\dd\mu^{d-1} =d\|R_{d-1}f-\1_E\|^2_{\textup{L}^2(E^c)}=o^{\varepsilon\to0}_d(1). \]
Using \eqref{eq:mainFbool} again, this yields
\[\mu^d(S\setminus\Omega)=\int_{\Omega^c}\1_S \le \big\|\1_S -R_df\big\|_{\textup{L}^1(\Omega^c)}+\int_{\Omega^c}R_df\dd\mu^d=o^{\varepsilon\to0}_d(1),\]
 so that $\mu^d(\Omega)\geq1-o^{\varepsilon\to0}_d(1).$
The geometric Loomis--Whitney inequality gives the matching upper bound
\[ \mu^d(\Omega) \leq \prod_{i=1}^d\mu^{d-1}(E)^{1/(d-1)} = 1+o^{\varepsilon\to0}_d(1). \]
Since both $\mu^d(S)$ and $\mu^d(\Omega)$ are $1+o^{\varepsilon\to0}_d(1)$, the preceding estimate for $\mu^d(S\setminus\Omega)$ also implies that $\mu^d(S\mathbin{\triangle}\Omega)=o^{\varepsilon\to0}_d(1)$.

Now as $\mu^d(\Omega)=1+o^{\varepsilon\to0}_d(1)$ and $\mu^{d-1}(E)=1-o^{\varepsilon\to0}_d(1)$, Lemma~\ref{lem:cylLWst} provides a set $B=A_1\times\cdots\times A_d\subseteq G^d$ such that $\mu^d(\Omega \mathbin{\triangle}B) = o^{\varepsilon\to0}_d(1)$. By the triangle inequality and \eqref{eq:mainFbool}, we find that
\begin{equation}\label{eq:mainFbox}
\begin{aligned}
\mu^d(B) =1+o^{\varepsilon\to0}_d(1),\quad \|R_df-\1_B\|_{\textup{L}^1(G^d)} + \|R_df-\1_B\|^2_{\textup{L}^2(G^d)} & = o^{\varepsilon\to0}_d(1).
\end{aligned}
\end{equation}
Writing
\[ B^{(i)}:=A_1\times\cdots\times A_{i-1}\times A_{i+1}\times\cdots\times A_d, \]
we can use the marginal identity \eqref{eq:Rmarg} in the $\textup{L}^1$ estimate in \eqref{eq:mainFbox} to obtain
\begin{equation*} \|(R_{d-1}f)^2-\mu(A_i)\1_{B^{(i)}}\|_{\textup{L}^1(G^{d-1})} = o^{\varepsilon\to0}_d(1). \end{equation*}
Since $(s^{1/2}-t^{1/2})^2\leq|s-t|$, this also yields
\begin{equation}\label{eq:marginalised}\|R_{d-1}f-\mu(A_i)^{1/2}\1_{B^{(i)}}\|^2_{\textup{L}^2(G^{d-1})} = o^{\varepsilon\to0}_d(1). \end{equation}
The $\textup{L}^2$ norms of the individual functions of \eqref{eq:marginalised} are $1+o^{\varepsilon\to0}_d(1)$, so multiplying out the square, we find the inner product of the cross term must also be $1+o^{\varepsilon\to0}_d(1)$. On the other hand, by \eqref{eq:mainnorm}, we find that
\[ \mu(A_i)^{1/2}\int_{B^{(i)}} R_{d-1}f \dd\mu^{d-1} \leq \mu(A_i)^{1/2} \int_{G^{d-1}}R_{d-1}f\dd\mu^{d-1} = \mu(A_i)^{1/2},\]
so that $\mu(A_i)\geq1+o^{\varepsilon\to0}_d(1)$. Since
\[ \prod_{i=1}^d \mu(A_i)=\mu^d(B)=1+o^{\varepsilon\to0}_d(1), \]
we can conclude that
\begin{equation}\label{eq:mainside}
\mu(A_i)=1+o^{\varepsilon\to0}_d(1)\quad\text{for } 1\leq i\leq d.
\end{equation}

Now in the $d=2$ case, by the triangle inequality and \eqref{eq:marginalised}, we have
\[
\|\mu(A_1)^{1/2}\1_{A_2}-\mu(A_2)^{1/2}\1_{A_1}\|^2_{\textup{L}^2(G)}= o^{\varepsilon\to0}(1).
\]
Multiplying out the square and using \eqref{eq:mainside} we find that
$
\mu(A_1\cap A_2)=1+o^{\varepsilon\to0}(1),
$
and so $\mu(A_1\mathbin{\triangle}A_2)=o^{\varepsilon\to0}(1)$. Since $R_1f$ is even, we can change variables in \eqref{eq:marginalised} and argue similarly to find $\mu(A_1\mathbin{\triangle}(-A_1))=o^{\varepsilon\to0}(1)$. Thus, defining $K:=A_1\cap(-A_1)$, we have $K=-K$, and we maintain the property $\mu(K)=1+o^{\varepsilon\to0}(1)$. By the triangle inequality and \eqref{eq:mainFbox} and \eqref{eq:marginalised}, we also have
\begin{equation}\label{eq:forlater}
\|R_2f-\1_{K\times K}\|_{\textup{L}^1(G^2)} + \|R_2f-\1_{K\times K}\|^2_{\textup{L}^2(G^2)} =o^{\varepsilon\to0}(1),
\end{equation}
\begin{equation}\label{eq:sltens}
\begin{aligned}
\|R_1f-\1_K\|^2_{\textup{L}^2(G)} & =o^{\varepsilon\to0}(1).
\end{aligned}
\end{equation}
By multiplying out the square in \eqref{eq:sltens} we find that $\int_K R_1f=1+o^{\varepsilon\to0}(1)$, and since
\[ \int_G R_1f\dd\mu = \Big(\int_Gf\dd\mu\Big)^2 = 1, \]
we can conclude that
\begin{equation}\label{eq:corrtail}
\int_{K^c}R_1f\dd\mu = o^{\varepsilon\to0}(1).
\end{equation}
Now applying Lemma~\ref{lem:diagleak} with $A=K^c$, combining with \eqref{eq:mainnorm} and \eqref{eq:corrtail}, we find
\[ \iint_{\{h_1+h_2\notin K\}} R_2f(h_1,h_2)\, \dd\mu(h_1)\dd\mu(h_2) = o^{\varepsilon\to0}(1). \]
 Noting that $\mu^2 \bigl(\{(h_1,h_2)\in K^2 : h_1+h_2\notin K\}\bigr)$ is bounded by
\begin{align*}
&\le 2\iint_{\{h_1+h_2\notin K\}} R_2f(h_1,h_2)\, \dd\mu(h_1)\dd\mu(h_2)+\mu^2 \bigl(\{(h_1,h_2)\in K^2 : R_2f(h_1,h_2)\le 1/2\}\bigr)\\
&\le 2\iint_{\{h_1+h_2\notin K\}} R_2f(h_1,h_2)\, \dd\mu(h_1)\dd\mu(h_2)+2\|R_2f-\1_{K\times K}\|_{\textup{L}^1(G^2)}
\end{align*}
and recalling \eqref{eq:forlater} and \eqref{eq:corrtail}, we conclude that
\[ \mu^2 \bigl(\{(h_1,h_2)\in K^2 : h_1+h_2\notin K\}\bigr) = o^{\varepsilon\to0}(1). \]

When $d\ge3$, we prove a similar estimate by taking a two-dimensional slice. Indeed,
for $z\in A_3\times\cdots\times A_d \subseteq G^{d-2}$ let
\[ \Lambda_1(z):=\|R_df(\,\cdot\,,z)-\1_{A_1\times A_2}\|_{\textup{L}^1(G^2)},
\quad \Lambda_2(z):=\|R_df(\cdot\,,z)-\1_{A_1\times A_2}\|_{\textup{L}^2(G^2)}^2. \]
By \eqref{eq:mainFbox} we have
\[ \int_{A_3\times\cdots\times A_d}(\Lambda_1+\Lambda_2)\dd\mu^{d-2} = o^{\varepsilon\to0}_d(1), \]
so by Chebyshev's inequality and \eqref{eq:mainside},
\[\mu^{d-2}\big(\{\,z\in A_3\times\cdots\times A_d\,: \, \Lambda_1(z)+\Lambda_2(z)<\lambda\}\big)\ge 1-\frac{o^{\varepsilon\to0}_d(1)}{\lambda}\]
whenever $\lambda>0$ is small.
Thus there is a set $D\subseteq A_3\times\cdots\times A_d$ of measure $1-o^{\varepsilon\to0}_d(1)$ on which $\Lambda_1+\Lambda_2=o^{\varepsilon\to0}_d(1)$. Recalling that
\[ \int_{G^{d-2}}R_{d-2}f(z)^2\dd\mu^{d-2}(z) = \|f\|_{\textup{U}^{d-1}(G)}^{2^{d-1}}=1, \]
since the integrand is nonnegative, its integral over $D$ is at most one. Hence there must be some $z\in D$ that satisfies
\begin{equation}\label{eq:goodsl}
\Lambda_1(z)+\Lambda_2(z) = o^{\varepsilon\to0}_d(1),
\quad R_{d-2}f(z)^2 \leq\mu^{d-2}(D)^{-1} = 1+o^{\varepsilon\to0}_d(1).
\end{equation}
For such a fixed $z$ satisfying \eqref{eq:goodsl}, define
\[ g(x):=\prod_{\omega\in\{0,1\}^{d-2}} f(x+\omega\cdot z), \]
so that
\begin{equation}\label{eq:identities} \int_Gg\dd\mu=R_{d-2}f(z), \quad R_1g=R_{d-1}f(\cdot,z),\quad R_2g=R_df(\cdot,z).
\end{equation}
Combining \eqref{eq:mainside}, \eqref{eq:goodsl} and \eqref{eq:identities} we find that
\[ \|g\|_{\textup{L}^1(G)}\le 1+o^{\varepsilon\to0}_d(1), \quad \|g\|_{\textup{U}^{2}(G)}=1+o^{\varepsilon\to0}_d(1),
\quad \|g\|_{\textup{U}^3(G)}=1+o^{\varepsilon\to0}_d(1). \]
Along with the fact that $\Lambda_1(z)+\Lambda_2(z) = o^{\varepsilon\to0}_d(1)$, these are the ingredients we need to repeat the $d=2$ argument, with $f$ replaced by $g$. Repeating the argument yields a symmetric $K$ such that
\begin{equation}\label{eq:pol}\mu(K)=1+o^{\varepsilon\to0}_d(1),\quad \mu^2 \bigl(\{(h_1,h_2)\in K^2 : h_1+h_2\notin K\}\bigr) = o^{\varepsilon\to0}_d(1). \end{equation}

Finally we note that \eqref{eq:pol} implies that
\[ \big\langle\1_K\ast \1_K,\1_K\big\rangle_{\textup{L}^2(G)} = 1 - o^{\varepsilon\to0}_d(1). \]
Then, by the Cauchy--Schwarz inequality, we obtain
\begin{equation}\label{eq:mainYnear}
\|\1_K\ast \1_K\|_{\textup{L}^2(G)} \geq \bigl(1-o^{\varepsilon\to0}_d(1)\bigr) \|\1_K\|_{\textup{L}^{4/3}(G)}^2.
\end{equation}
Letting $\eta_{\mathrm{Young}}$ denote the absolute constant in Lemma~\ref{lem:unifYgap}, we choose $\varepsilon_d\in(0,1)$ so small that the coefficient $1-o^{\varepsilon\to0}_d(1)$ in \eqref{eq:mainYnear} is greater than $1-\eta_{\mathrm{Young}}$ whenever $0<\varepsilon\leq\varepsilon_d$. If $G$ has no compact open subgroup, the lemma contradicts \eqref{eq:mainYnear}. Hence, with the normalisation \eqref{eq:mainnorm}, we have $\|f\|_{\textup{U}^{d}(G)}<1-\varepsilon_d$, and so by undoing the normalisation we obtain the improved log-convexity estimate \eqref{eq:mainlog}.
\end{proof}

As we mentioned earlier, there are alternative ways of obtaining a contradiction from the existence of a set $K$ that satisfies \eqref{eq:pol}, such as a variant of \cite[Lemma\,5.5]{EisnerTao}. We chose the more classical sharpened Young inequality to retain the analogy with the proof of Theorem~\ref{thm:sharpRn}(b).


\section{One-dimensional obstructions}
\label{sec:realobs}

We next record two observations specific to the real line $G=\R$ with $d=2$. They explain why it is unlikely that Theorem~\ref{thm:mainlog} can be established by studying only the simplest autocorrelation function $R_1 f = f\ast\widetilde{f}$. The authors admit to having spent considerable time pursuing this route and therefore find it useful to record the obstructions. Consequently, when specialised to $d=2$, the proof had to be essentially ``two-dimensional'': it also had to consider the multiple autocorrelation function $R_2 f$.

More precisely, if
\[ f\geq0, \quad \|f\|_{\textup{U}^1(\R)}=1, \quad \|f\|_{\textup{U}^3(\R)}=1, \]
and
\[ \|f\|_{\textup{U}^2(\R)} \text{ is sufficiently close to } 1, \]
then the preceding arguments readily show that $f\ast\widetilde{f}$ must be close to the indicator function~$\1_E$ of a symmetric set $E\subset\R$ of measure $1$ in both the $\textup{L}^1$ and $\textup{L}^2$ norms; cf.\@ the proof of \cite[Prop.\,2.5]{BennettTao} or recall \eqref{eq:mainFbool} (which is an overkill).
We will see that this closeness in $\textup{L}^1(\R)$ alone is insufficient to yield a contradiction, whereas an argument based on closeness in $\textup{L}^2(\R)$ runs into a well-known open problem.

\subsection{Obstruction to the \texorpdfstring{$\textup{L}^1$}{L1} route}

Here is a construction showing that the autocorrelation $f\ast\widetilde{f}$ of a nonnegative function $f$ of total mass $1$ can be arbitrarily close to an indicator function in the $\textup{L}^1$ sense.
Let
\[ \operatorname{dist}\bigl(t,\{0,1\}\bigr) = \begin{cases}
|t| & \text{if } t\leq1/2, \\
|t-1| & \text{if } t>1/2 \\
\end{cases} \]
be the distance of a real number $t$ to the set $\{0,1\}$.
For a measurable function $f\colon\R\to\R$, put
\[ \mathcal{E}(f):=\int_\R \operatorname{dist}\bigl(f(x),\{0,1\}\bigr)\dd x. \]
This functional measures how close $f$ is to an indicator function in the norm of $\textup{L}^1(\R)$.

\begin{proposition}
There exists a sequence $(f_k)_{k=0}^\infty$ of bounded compactly supported measurable functions $f_k\colon\R\to[0,\infty)$, each having integral $1$, such that
\[ \lim_{k\to\infty} \mathcal{E}\bigl(f_k\ast\widetilde{f_k}\bigr) = 0. \]
\end{proposition}

Consequently, for every $\varepsilon>0$ there exist a measurable set $E\subseteq\R$ and a measurable function $f$ such that
\[ -E=E, \quad |E|=1, \quad f\geq0,\quad \int_{\R}f=1,\quad \text{and}\quad \bigl\|f\ast\widetilde{f}-\1_E\bigr\|_{\textup{L}^1(\R)}<\varepsilon. \]
This fact might seem counterintuitive at first, because the positive semidefinite function $f\ast\widetilde{f}$ necessarily has a ``tall peak'' around $0$, but note that we are measuring the proximity in the $\textup{L}^1$ norm. The first two steps of the construction in the following proof are illustrated in Figures \ref{fig:gowers_auto1} and \ref{fig:gowers_auto2}, produced with Mathematica~\cite{Mathematica}.

\begin{figure}
\begin{center}
\includegraphics[width=0.8\linewidth]{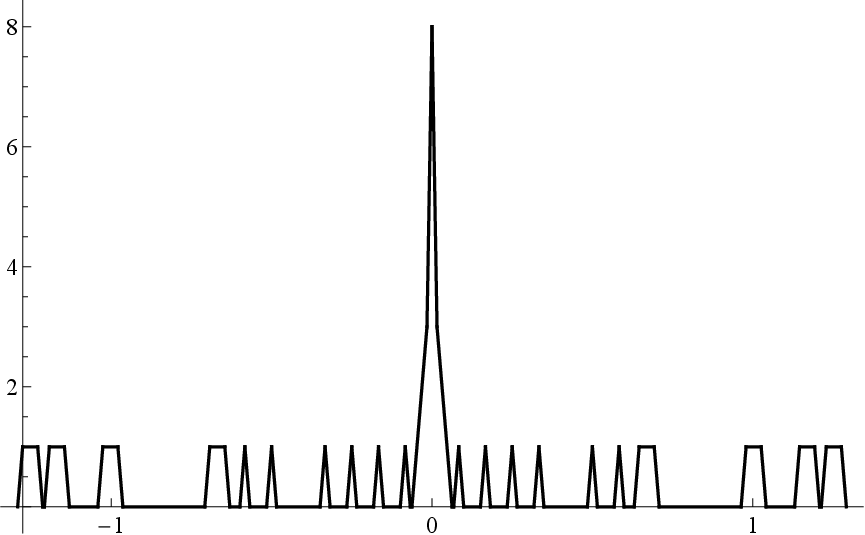}
\end{center}
\caption{Graph of $f_1\ast\widetilde{f_1}$.}
\label{fig:gowers_auto1}
\end{figure}

\begin{figure}
\begin{center}
\includegraphics[width=0.8\linewidth]{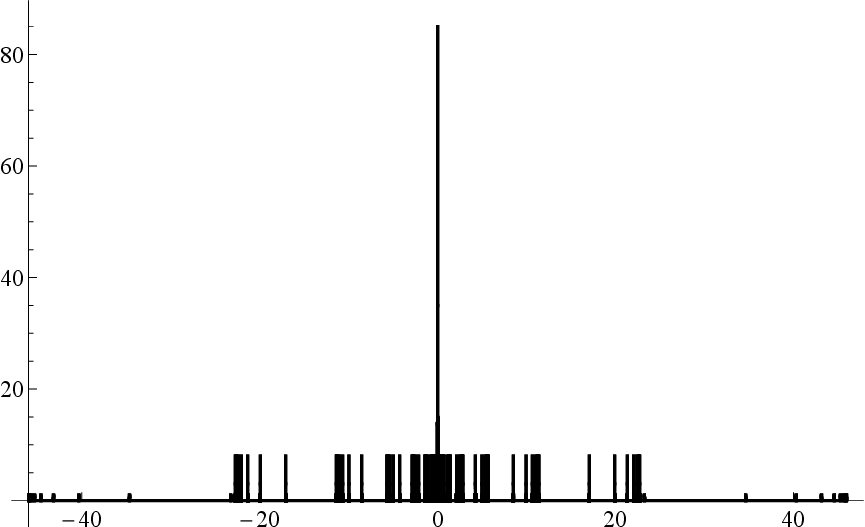}\\
\includegraphics[width=0.8\linewidth]{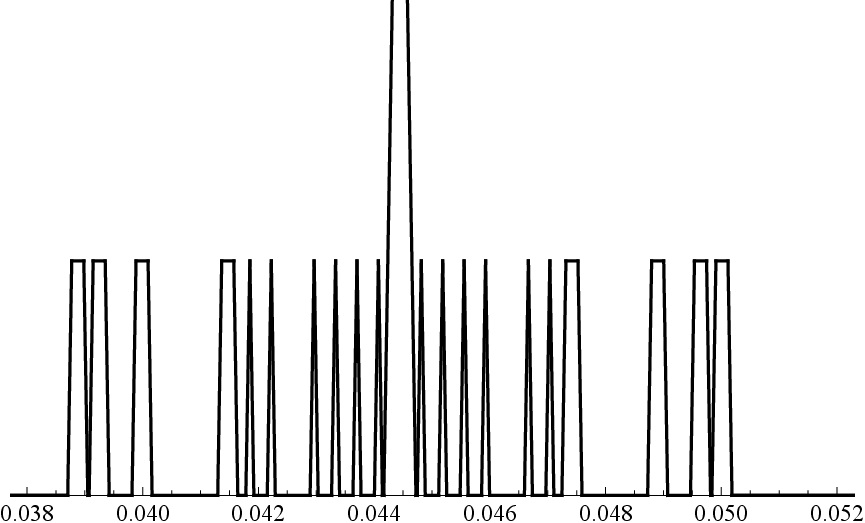}
\end{center}
\caption{Graph of $f_2\ast\widetilde{f_2}$ (top) and an enlarged view around the point $(0.045,1)$ (bottom).}
\label{fig:gowers_auto2}
\end{figure}

\begin{proof}
We construct the desired sequence $(f_k)_{k=0}^\infty$ inductively. Let $f_0:=\1_{[0,1]}$, and suppose that $f_k$ is already constructed for some index $k\geq0$, is
supported in an interval of length $D$ and has integral 1.
First define
\[ g_k(x) := \sum_{j=1}^N N f_k\bigl(N^2(x-M\,2^j)\bigr), \]
where $M>0$ and $N\in\N$ are parameters to be chosen later. Note that $g_k$ still has integral $1$ and its autocorrelation can be expanded as
\[ \bigl(g_k\ast\widetilde{g_k}\bigr)(h) = N \bigl(f_k\ast\widetilde{f_k}\bigr)(N^2 h) + \sum_{i\ne j} \bigl(f_k\ast\widetilde{f_k}\bigr)\Bigl(N^2\bigl(h-M(2^j-2^i)\bigr)\Bigr). \]
For sufficiently large $M$, the diagonal term and all the off-diagonal summands above have pairwise disjoint supports, so
\begin{equation}\label{eq:gk_auto}
\mathcal{E}\bigl(g_k\ast\widetilde{g_k}\bigr) \leq\Bigl(1-\frac{1}{N}\Bigr)\mathcal{E}\bigl(f_k\ast\widetilde{f_k}\bigr) + \frac{1}{N}.
\end{equation}
Indeed, the $N(N-1)$ off-diagonal terms contribute $(1-1/N)\mathcal{E}(f_k\ast\widetilde{f_k})$, while the diagonal contribution is at most $1/N$.

Next, we define
\[ f_{k+1}(x) := \frac{1}{\alpha}g_k\Bigl(\frac{x}{\alpha^2}\Bigr) + \frac{N}{\alpha} \1_{[L,L+\alpha\beta/N]}(x), \]
for some $\alpha,\beta>0$ with $\alpha+\beta=1$ and some $L>0$, which will also be chosen later.
Note that $f_{k+1}$ integrates to $1$ and
\begin{align}
\bigl(f_{k+1}\ast\widetilde{f_{k+1}}\bigr)(h)
& = \bigl(g_k\ast\widetilde{g_k}\bigr)\Bigl(\frac{h}{\alpha^2}\Bigr) \label{eq:auto1} \\
& \quad + N \bigl(\widetilde{g_k}\ast\1_{[L/\alpha^2,L/\alpha^2+\beta/(\alpha N)]}\bigr)\Bigl(\frac{h}{\alpha^2}\Bigr) \label{eq:auto2} \\
& \quad + N \bigl(g_k\ast\1_{[-L/\alpha^2-\beta/(\alpha N),-L/\alpha^2]}\bigr)\Bigl(\frac{h}{\alpha^2}\Bigr) \label{eq:auto3} \\
& \quad + \frac{N^2}{\alpha^2}\bigl(\1_{[0,\alpha\beta/N]}\ast\1_{[-\alpha\beta/N,0]}\bigr)(h). \label{eq:auto4}
\end{align}
Taking $L$ sufficiently large makes the supports of the terms in \eqref{eq:auto1}--\eqref{eq:auto3} mutually disjoint, so we can estimate their contributions to $\mathcal{E}$ separately.
By shrinking \eqref{eq:gk_auto} by the factor $\alpha^2$, we can estimate $\mathcal{E}$ of \eqref{eq:auto1} by
\[ \leq \alpha^2 \Bigl(1-\frac{1}{N}\Bigr) \mathcal{E}\bigl(f_k\ast\widetilde{f_k}\bigr) + \frac{\alpha^2}{N}. \]
Next, the convolution $N\widetilde{g_k}\ast\1_{[L/\alpha^2,L/\alpha^2+\beta/(\alpha N)]}$ appearing in \eqref{eq:auto2} can be expanded as
\[ \sum_{j=1}^N N^2 f_k\bigl(-N^2(\,\cdot\,+M\,2^j)\bigr) \ast\1_{[L/\alpha^2,L/\alpha^2+\beta/(\alpha N)]}. \]
Each term is a convolution of the indicator of the interval $[L/\alpha^2,L/\alpha^2+\beta/(\alpha N)]$ with a nonnegative function of total mass $1$ supported on an interval of length $D/N^2$. Such a convolution can differ from $0$ or $1$ only on two intervals of length $D/N^2$ each and the condition $N\geq 10\alpha D/\beta$ ensures that these intervals are disjoint. By increasing $M$ if necessary, the $N$ convolutions corresponding to distinct values of $j$ have pairwise disjoint supports. After scaling by $\alpha^2$, we therefore bound the contribution of \eqref{eq:auto2} to $\mathcal{E}$ by $2\alpha^2 D/N$.
The term \eqref{eq:auto3} is then handled in exactly the same way.
Finally, the last term \eqref{eq:auto4} can only perturb the value of $\mathcal{E}$ by at most its $\textup{L}^1$ norm, which is $\beta^2$.
Adding all these contributions gives
\begin{equation}\label{eq:L1step}
\mathcal{E}\bigl(f_{k+1}\ast\widetilde{f_{k+1}}\bigr) \leq \alpha^2 \Bigl(1-\frac{1}{N}\Bigr) \mathcal{E}\bigl(f_k\ast\widetilde{f_k}\bigr) + \frac{\alpha^2(1+4D)}{N} + \beta^2 .
\end{equation}

Finally we adjust the parameters in this inductive construction.
We set
\[ \alpha=\frac{k+1}{k+2}, \quad \beta=\frac{1}{k+2}, \]
and choose $N$ sufficiently large (depending on $k$ and $D$) as before, but also so that
\[ \frac{\alpha^2(1+4D)}{N} \leq \beta^4. \]
We then choose $M$, and finally $L$, as required by the preceding arguments.
In this way, \eqref{eq:L1step} gives us
\[ \mathcal{E}\bigl(f_{k+1}\ast\widetilde{f_{k+1}}\bigr) \leq \Bigl(\frac{k+1}{k+2}\Bigr)^2 \mathcal{E}\bigl(f_k\ast\widetilde{f_k}\bigr) + \frac{1}{(k+2)^2}+\frac{1}{(k+2)^4}. \]
Multiplying the estimate by $(k+2)^2$, we obtain
\[ (k+2)^2 \mathcal{E}\bigl(f_{k+1}\ast\widetilde{f_{k+1}}\bigr) \leq (k+1)^2 \mathcal{E}\bigl(f_k\ast\widetilde{f_k}\bigr) + 1 +\frac{1}{(k+2)^2}, \]
which can be iterated to find that
\[ (k+1)^2 \mathcal{E}\bigl(f_k\ast\widetilde{f_k}\bigr) = O^{k\to\infty}(k). \]
Thus we see that $\mathcal{E}(f_k\ast\widetilde{f_k}) = O^{k\to\infty}(1/k)$.
\end{proof}

\subsection{Obstruction to the \texorpdfstring{$\textup{L}^2$}{L2} route}

Passing to the Fourier side and using Plancherel's theorem, we get
\[ \bigl\|f\ast\widetilde{f}-\1_{E}\bigr\|_{\textup{L}^2(\R)}
= \Bigl\| \bigl|\widehat{f}\,\bigr|^2 - \widehat{\1_{E}} \Bigr\|_{\textup{L}^2(\R)}
\geq \bigl\| (\widehat{\1_E})_- \bigr\|_{\textup{L}^2(\R)}. \]
Consequently, Theorem~\ref{thm:mainlog} in the case $G=\R$ and $d=2$ would follow if we could show
\begin{equation}\label{eq:Fourgap}
\inf_{\substack{E\text{ measurable}\\-E=E,\ |E|=1}} \bigl\|(\widehat{\1_E})_-\bigr\|_{\textup{L}^2(\R)}>0.
\end{equation}
This has a discrete reformulation. For a nonempty symmetric finite set $A\subset\Z$ put
\[ P_A(t):=\sum_{a\in A}e^{-2\pi iat},
\quad E_A:=\bigcup_{a\in A}\Bigl[\frac{a-1/2}{|A|},\frac{a+1/2}{|A|}\Bigr]. \]
Then $|E_A|=1$ and
\[ \widehat{\1_{E_A}}(\xi) = \frac{1}{|A|}\operatorname{sinc}\Bigl(\frac{\xi}{|A|}\Bigr) P_A\Bigl(\frac{\xi}{|A|}\Bigr), \]
so that
\[ \bigl\|(\widehat{\1_{E_A}})_-\bigr\|_{\textup{L}^2(\R)}^2 =\frac{1}{|A|}\int_\R \bigl(\operatorname{sinc}(t)P_A(t)\bigr)_-^2\dd t. \]
Thus, \eqref{eq:Fourgap} implies that there exists an absolute constant $c>0$ such that
\begin{equation}\label{eq:wcosgap}
\int_\R\bigl(\operatorname{sinc}(t)P_A(t)\bigr)_-^2\dd t \geq c|A|
\end{equation}
holds for every nonempty finite $A\subset\Z$ with $-A=A$ and $0\notin A$.
It is also easy to use a standard approximation argument to justify that \eqref{eq:Fourgap} and \eqref{eq:wcosgap} are, in fact, equivalent.

A further natural sufficient condition for \eqref{eq:Fourgap} is the unweighted inequality
\begin{equation}\label{eq:qChowla}
\int_{\T} \bigl(P_A(t)\bigr)_-^2 \dd t\geq c|A|.
\end{equation}
Indeed, $\operatorname{sinc}(t)\geq2/\pi$ on $[-1/2,1/2]$, so \eqref{eq:qChowla} implies \eqref{eq:wcosgap}.
By contrast, the work of McGehee, Pigno and Smith \cite{GPS81} and Konyagin \cite{Konyagin1981} yields only
\[ \int_\T \bigl(P_A(t)\bigr)_-^2 \dd t \gtrsim (\log|A|)^2 \]
(which follows from their solution of Littlewood's conjecture, that $\|P_A\|_{\textup{L}^1(\T)} \gtrsim \log|A|$, combined with the fact that $\int_\T P_A=0$).
This is still much weaker than \eqref{eq:qChowla}.
In fact, \eqref{eq:qChowla} would also settle the conjectured sharp order of magnitude in Chowla's cosine problem \cite{Chowla1965}: for a symmetric finite set $A\subset\Z\setminus\{0\}$ one has
\[ P_A(t)=\sum_{a\in A}\cos(2\pi at), \]
and \eqref{eq:qChowla} would force $P_A$ to attain a negative value of magnitude at least a constant multiple of $|A|^{1/2}$.
There has been some very recent progress on Chowla's cosine problem \cite{Bedert25,JMTZ25}, but even the conjectured bound in that problem would not by itself imply \eqref{eq:qChowla}.


\section{Sharp estimates on discrete cubes}
\label{sec:dcube}

This section is devoted to the proof of Theorem~\ref{thm:dcres}. The following product principle is the analogue of \cite[Prop.\,2]{BCK25}, which, in turn, was a modification of \cite[Prop.\,21]{DGIM21}.

\begin{proposition}\label{prop:dcprod}
Let $q,t>0$ satisfy $qt=2^d p$. The following assertions are equivalent:
\begin{enumerate}[(a)]
\item Every nonnegative $f\colon\Z\to[0,\infty)$ supported in $\{0,1,\ldots,n-1\}$ satisfies
\[ [f]_{\textup{U}^{d,p}(\Z)}\leq \|f\|_{\ell^q(\Z)}. \]
\item For every $m\geq0$, every nonnegative $f$ supported in $\{0,1,\ldots,n-1\}^m$ satisfies
\[ [f]_{\textup{U}^{d,p}(\Z^m)}\leq \|f\|_{\ell^q(\Z^m)}. \]
\item For every $m\geq0$ and every $A\subseteq\{0,1,\ldots,n-1\}^m$ one has $\mathcal{P}_{d,p}(A)\leq|A|^t$.
\item There is a finite constant $C>0$ such that for every $m\geq0$ and every $A\subseteq\{0,1,\ldots,n-1\}^m$ one has $\mathcal{P}_{d,p}(A)\leq C|A|^t$.
\end{enumerate}
Consequently, the extremal exponents $q_{d,p,n}$ and $t_{d,p,n}$ are attained and \eqref{eq:dcrel} holds.
\end{proposition}

\begin{proof}
We will be rather brief, because the proof follows the same outline as the proofs in \cite{DGIM21} and \cite{BCK25}.
In the discrete setting, Lemma~\ref{lm:UdpGCS} from Appendix \ref{sec:Udp_are_norms} reads
\begin{align}
\sum_{x_1,\ldots,x_p,h_1,\ldots,h_d\in\Z^m} \prod_{j=1}^p\prod_{\omega=(\omega_1,\ldots,\omega_d)\in\{0,1\}^d} f_{j,\omega}(x_j+\omega_1 h_1+\cdots+\omega_d h_d) & \label{eq:dcmix} \\[-2mm]
\leq\, \prod_{j=1}^p\prod_{\omega\in\{0,1\}^d} &[f_{j,\omega}]_{\textup{U}^{d,p}(\Z^m)}.\nonumber
\end{align}
Here, all the functions are nonnegative and supported on finite subsets of $\Z^m$.

We prove (a)$\Rightarrow$(b) by induction on $m$. Slice a function on $\Z^m$ in its last coordinate and write $f_b(a) := f(a,b)$. Expanding the power $[f]_{\textup{U}^{d,p}(\Z^m)}^{2^d p}$ and applying \eqref{eq:dcmix} in the first $m-1$ coordinates bounds it by
\[ \sum_{b_1,\ldots,b_p,l_1,\ldots,l_d\in\Z} \prod_{j=1}^p\prod_{\omega=(\omega_1,\ldots,\omega_d)\in\{0,1\}^d} [f_{b_j+\omega_1 l_1+\cdots+\omega_d l_d}]_{\textup{U}^{d,p}(\Z^{m-1})}. \]
This is the power $2^d p$ of the one-dimensional functional applied to $b\mapsto[f_b]_{\textup{U}^{d,p}(\Z^{m-1})}$. Assertion (a), followed by the induction hypothesis on each slice, therefore gives
\[ [f]_{\textup{U}^{d,p}(\Z^m)} \leq \Bigl(\sum_b [f_b]_{\textup{U}^{d,p}(\Z^{m-1})}^q\Bigr)^{1/q} \leq\|f\|_{\ell^q(\Z^m)}. \]

The implications (b)$\Rightarrow$(c)$\Rightarrow$(d) follow by taking $f=\1_A$ and $C=1$.

It remains to prove (d)$\Rightarrow$(a). By Corollary \ref{cor:Udpnorms} below, the map $f\mapsto[|f|]_{\textup{U}^{d,p}(\Z^m)}$ is a norm. A dyadic decomposition from \cite{DGIM21} or \cite{BCK25} writes $f$ as a sum of $O_{d,p,n}^{m\to\infty}(1+m)$ scalar multiples of indicators of level sets and a remainder that can be controlled by Theorem~\ref{thm:Lpest}(a). This decomposition, combined with (d), gives
\begin{equation}\label{eq:cubes_aux}
[f]_{\textup{U}^{d,p}(\Z^m)} \leq O_{d,p,n,C}^{m\to\infty}(1+m) \|f\|_{\ell^q(\Z^m)}.
\end{equation}
Now we apply \eqref{eq:cubes_aux} to the $m$-fold tensor power of a fixed one-dimensional function $f$. The polynomial loss disappears after taking the $m$th root and letting $m\to\infty$. This proves (a).

It remains to note that condition (a) is finite-dimensional and depends continuously on $q$, so the maximal such exponent $q=q_{d,p,n}$ exists.
\end{proof}

We now prove the three assertions of Theorem~\ref{thm:dcres}. We may assume that $p$ is an integer and $p\geq2$, since the case $p=1$ was already studied in \cite{BCK25}.

\begin{proof}[Proof of Theorem~\ref{thm:dcres}(a)]
For a nonnegative function $f$ supported on $\{0,1\}$, a direct enumeration of the terms appearing in
\[ \|R_d f\|_{\ell^p(\Z^d)}^p = \sum_{h_1,\ldots,h_d\in\Z} R_d f (h_1,\ldots,h_d)^p \]
gives
\[ [f]_{\textup{U}^{d,p}(\Z)}^{2^d p} = \bigl(f(0)^{2^d}+f(1)^{2^d}\bigr)^p + 2d\bigl(f(0)f(1)\bigr)^{p\,2^{d-1}}. \]
Put $t=\log_2(2^p+2d)$. By Proposition \ref{prop:dcprod}, after setting $x=f(0)^{2^d p/t}$ and $y=f(1)^{2^d p/t}$, it is enough to prove
\begin{equation}\label{eq:dcscalar}
\bigl(x^{t/p}+y^{t/p}\bigr)^p+2d(xy)^{t/2} \leq(x+y)^t
\end{equation}
for $x,y\geq0$.
The cases $x=0$ and $y=0$ of \eqref{eq:dcscalar} hold with equality, so by homogeneity and symmetry, we may suppose that $x\geq y$ and $x+y=1$, and substitute
\[ x=\frac{e^v}{2\cosh v},\quad y=\frac{e^{-v}}{2\cosh v}\quad\text{for } v\geq0. \]
The desired estimate becomes
\[ \Bigl(2\cosh\frac{tv}{p}\Bigr)^p + 2d \leq (2\cosh v)^t, \]
i.e., after taking into account the definition of $t$ and making the substitution $u=tv$,
\[ 2^t \biggl( \Bigl(\cosh \frac{u}{t}\Bigr)^t - 1 \biggr) \geq 2^p \biggl( \Bigl(\cosh\frac{u}{p}\Bigr)^p - 1 \biggr) . \]
The proof of \cite[Lem.\,12]{CKS25} begins by showing that the function $s\mapsto 2^s ((\cosh (u/s))^s-1)$ is increasing on $[2,\infty)$ for every fixed $u>0$. Since $t>p\geq2$, this implies the last displayed inequality, and thus also establishes \eqref{eq:dcscalar}.

Sharpness follows from considering the full binary cube:
\[ \mathcal{P}_{d,p}(\{0,1\}^m)=(2^p+2d)^m. \]
Thus $t_{d,p,2}=\log_2(2^p+2d)$, and Proposition \ref{prop:dcprod} also gives the formula for $q_{d,p,2}$.
\end{proof}

\begin{proof}[Proof of Theorem~\ref{thm:dcres}(b)]
We first note that $t_{d,p,n}\to d+p$ as $n\to\infty$.
Namely, Theorem~\ref{thm:Lpest}(a) gives $t_{d,p,n}\leq d+p$, while testing \eqref{eq:dcset} on the full interval of integers gives that $\mathcal{P}_{d,p}(\{0,1,\ldots,n-1\})$ is at least $n^{d+p}$ times a constant depending on $d$ and $p$.
Consequently, $t_{d,p,n}=d+p+O_{d,p}^{n\to\infty}(1/\log n)$.

For $M>1$, define a truncated discrete Gaussian by
\[ \varphi_{M,n}(j)
:=\begin{cases}
\exp\bigl(-4M^2(j/n-1/2)^2\bigr), & 0\leq j<n,\\
0, & \text{otherwise}.
\end{cases} \]
As in \cite{Shao} or \cite{BCK25}, the convergence of Riemann sums gives
\begin{align*}
n^{-d-p}[\varphi_{M,n}]_{\textup{U}^{d,p}(\Z)}^{2^d p}
& \longrightarrow \bigl[e^{-4M^2(\cdot-1/2)^2}\1_{[0,1]}\bigr]_{\textup{U}^{d,p}(\R)}^{2^d p},\\
n^{-1}\|\varphi_{M,n}\|_{\ell^{q_{d,p,n}}(\Z)}^{q_{d,p,n}}
& \longrightarrow \int_0^1e^{-4M^2 2^d p(x-1/2)^2/(d+p)}\dd x,
\end{align*}
both as $n\to\infty$.
Apply the critical inequality \eqref{eq:ell_critical} to $\varphi_{M,n}$, take logarithms and pass to these limits. We obtain
\begin{align*}
\liminf_{n\to\infty} \bigl(t_{d,p,n}-d-p\bigr)\log_2n
\geq 2^d p\log_2 \frac{[e^{-x^2}\1_{[-M,M]}]_{\textup{U}^{d,p}(\R)}}{\|e^{-x^2}\1_{[-M,M]}\|_{\textup{L}^{2^d p/(d+p)}(\R)}}.
\end{align*}
As $M\to\infty$, the sharp Gaussian constant from Theorem~\ref{thm:sharpRn}(a) shows that the right-hand side tends to
\[ 2^d p\log_2 \Bigl(\frac{2^{2d}p^p}{(d+p)^{d+p}}\Bigr)^{1/(2^{d+1}p)}
= -\frac{(d+p)\log_2(d+p)-p\log_2p-2d}{2}. \]
This proves the lower estimate in \eqref{eq:dclargen}.

For the upper estimate, $0\leq R_d\1_A(h)\leq|A|$ gives
\[ \mathcal{P}_{d,p}(A)
\leq |A|^{p-1}\sum_hR_d\1_A(h)
= |A|^{p-1}\mathcal{P}_{d,1}(A). \]
Therefore $t_{d,p,n}\leq p-1+t_{d,1,n}$. The upper bound in \cite[Thm.\,3]{BCK25}, which in turn relied on the work of Shao \cite{Shao}, now proves the remaining estimate with the same absolute constant $c$.
\end{proof}

\begin{proof}[Proof of Theorem~\ref{thm:dcres}(c)]
We simply adapt the entropy argument from \cite[Sec.\,6]{BCK25}. For a finitely supported random variable $Y$, we write
\[ \operatorname{H}(Y) :=-\sum_y\mathbb{P}(Y=y)\log_2\mathbb{P}(Y=y) \]
for its \emph{Shannon entropy} measured in bits or shannons.
If $X_1,\ldots,X_l$ are independent Bernoulli random variables with parameter $1/2$, the $h_i$ are nonzero integers and $\sum_i |h_i|\leq n-1$, then
\begin{equation}\label{eq:dcent}
\frac{\operatorname{H}(h_1X_1+\cdots+h_lX_l)}{l}
\geq \frac{\operatorname{H}(\operatorname{B}(n-1,1/2))}{n-1},
\end{equation}
while equality is possible only when $l=n-1$ and $|h_1|=\cdots=|h_{n-1}|=1$; see \cite[Prop.\,10 and Cor.\,11]{BCK25}.

For the lower bound, apply the one-dimensional critical inequality to
the function
\[ f(j) := \Bigl(\frac{\binom{n-1}{j}}{2^{n-1}}\Bigr)^{t_{d,p,n}/(2^d p)} \quad\text{for } 0\leq j<n. \]
Its $\ell^{q_{d,p,n}}(\Z)$ norm is one. On the left-hand side, retain only configurations with exactly $n-1$ nonzero increments, all equal to $1$ or $-1$. There are $2^{n-1}\binom{d}{n-1}$ such configurations, and the factor $p$ cancels from every resulting exponent. Hence
\[ 2^{n-1}\binom{d}{n-1} 2^{-\operatorname{H}(\operatorname{B}(n-1,1/2))t_{d,p,n}} \leq1. \]
Taking logarithms gives
\[ \operatorname{H}(\operatorname{B}(n-1,1/2))t_{d,p,n} \geq(n-1)\log_2(2d)-\log_2((n-1)!) + \sum_{j=0}^{n-2}\log_2(1-j/d), \]
which is the required lower asymptotic.

For the upper bound, fix a small $\delta>0$ and set
\[ t = \frac{(n-1)\log_2(2d)-\log_2((n-1)!)}{\operatorname{H}(\operatorname{B}(n-1,1/2))}+\delta. \]
By Proposition \ref{prop:dcprod}, it is enough to prove the one-dimensional inequality
\[ \sum_{h_1,\ldots,h_d\in\Z}\bigl(R_df(h_1,\ldots,h_d)\bigr)^p \leq\Bigl(\sum_{j=0}^{n-1}f(j)^{2^d p/t}\Bigr)^t. \]
Substitute $g(j)=f(j)^{2^d p/t}$ and normalise so that $\sum_{j=0}^{n-1}g(j)=1$. Grouping the increments $h_i$ according to the number $l$ of nonzero ones reduces the left-hand side to
\begin{equation}\label{eq:dcgroups}
\sum_{l=0}^{n-1}\binom{d}{l} \sum_{\substack{h_1,\ldots,h_l\ne0\\
\sum_i |h_i|\leq n-1}} \Bigl(\sum_{a\in\Z} \prod_{\omega\in\{0,1\}^l} g(a+\omega\cdot h)^{t/(p\,2^l)}\Bigr)^p,
\end{equation}
where inadmissible summands are understood to be $0$. The term with $l=0$ is
\[ \Bigl(\sum_{j=0}^{n-1}g(j)^{t/p}\Bigr)^p. \]
When $l=n-1$, each admissible tuple $(h_1,\ldots,h_l)$ consists of signs $h_i\in\{-1,1\}$ and has a unique base point. The
resulting contribution is therefore identical to the dominant contribution in \cite[Sec.\,6]{BCK25}. When $1\leq l\leq n-2$, use Jensen's inequality for the $p$th power. Each resulting term is bounded by $n^{p-1}$ times
\[ \binom{d}{l}\Bigl(\prod_{j=0}^{n-1}g(j)^{\alpha_j}\Bigr)^t, \]
where $(\alpha_0,\ldots,\alpha_{n-1})$ is the distribution of a translate of $h_1X_1+\cdots+h_lX_l$. The weighted arithmetic--geometric mean inequality gives
\[ \prod_jg(j)^{\alpha_j} \leq2^{-\operatorname{H}(h_1X_1+\cdots+h_lX_l)}. \]
By \eqref{eq:dcent}, all terms with $1\leq l\leq n-2$ are $o_{n,p}^{d\to\infty}(1)$, while the total $l=n-1$ contribution is at most $2^{-\operatorname{H}(\operatorname{B}(n-1,1/2))\delta}$.

Finally, we argue as in \cite[Sec.\,6]{BCK25}. Fix
\[ \vartheta :=2^{-n\,2^n\operatorname{H}(\operatorname{B}(n-1,1/2))/(n-1)}. \]
If $g(j)\leq1-\vartheta$ for every $j$, then the $l=0$ term in \eqref{eq:dcgroups} is at most $n^p(1-\vartheta)^t=o_{n,p}^{d\to\infty}(1)$, and the preceding strict bound finishes the proof.
Otherwise, choose $j_0$ with $g(j_0)>1-\vartheta$. Since $t/p\geq2$ for sufficiently large $d$, the $l=0$ term satisfies
\[ \Bigl(\sum_jg(j)^{t/p}\Bigr)^p \leq\Bigl(\sum_jg(j)^2\Bigr)^{t/2} \leq\sum_jg(j)^2. \]
For all sufficiently large $d$, we also have $t/2^n\geq1$.
Every monomial arising from a nonzero tuple $(h_1,\ldots,h_l)$ has two distinct indices $j_1,j_2$ in its support, each carrying weight at least $2^{-l}\geq2^{-n+1}$. At least one of them differs from $j_0$, so the corresponding value of $g$ is at most $\vartheta$. After factoring out one factor of each of $g(j_1)$ and $g(j_2)$, the monomial is bounded by
\[ g(j_1)g(j_2)\vartheta^{t/2^n-1}. \]
After accounting for the binomial coefficients, the total number of terms is $O_{n,p}^{d\to\infty}(d^{n-1})$, and the choice of $\vartheta$ gives
\[ d^{n-1}\vartheta^{t/2^n-1} = O_{n,p,\delta}^{d\to\infty}(d^{-1}). \]
For sufficiently large $d$, their total contribution is therefore at most $2\sum_{i<j}g(i)g(j)$. Adding the $l=0$ term, we bound the expression in \eqref{eq:dcgroups} by
\[ \sum_jg(j)^2+2\sum_{i<j}g(i)g(j) = \Bigl(\sum_jg(j)\Bigr)^2=1. \]
We have proved the desired inequality for all sufficiently large $d$, so the upper asymptotic follows by letting $\delta\to 0$.
\end{proof}


\appendix

\section{Integer values of \texorpdfstring{$p$}{p} lead to norms}
\label{sec:Udp_are_norms}

This appendix is largely just an adaptation of the arguments by Shkredov \cite{Shkredov23}; we include it for completeness. A notable difference is that the definitions in that article do not apply literally to $[\,\cdot\,]_{\textup{U}^{d,p}}$ when $G$ is non-compact, since the additional averaging over $G$ is unavailable in that case.

Let $(G,+)$ be a second-countable LCA group and $\mu$ its Haar measure. The following lemma generalises the so-called Gowers--Cauchy--Schwarz inequality \cite{Gowers2001,HostKra05,EisnerTao}.

\begin{lemma}\label{lm:UdpGCS}
Fix integers $d\geq1$ and $p\geq1$. Let $(f_{j,\omega})$ be measurable functions on $G$ indexed by pairs $(j,\omega)\in\{1,\ldots,p\}\times\{0,1\}^d$. Suppose either that the functions $f_{j,\omega}$ are nonnegative or, in the complex-valued case, that $p$ is even and $[|f_{j,\omega}|]_{\textup{U}^{d,p}(G)}<\infty$ for every $(j,\omega)$. Then
\begin{align*}
\biggl| \int_{G^{p+d}} \prod_{j=1}^p \prod_{\omega=(\omega_1,\ldots,\omega_d)\in\{0,1\}^d} f_{j,\omega}\bigl(x_j + \omega_1 h_1 + \cdots + \omega_d h_d\bigr) & \\[-2mm]
\dd\mu(x_1) \cdots \dd\mu(x_p) \dd\mu(h_1) \cdots \dd\mu(h_d) & \biggr|
\leq \prod_{j=1}^p\prod_{\omega\in\{0,1\}^d} [f_{j,\omega}]_{\textup{U}^{d,p}(G)}.
\end{align*}
\end{lemma}

\begin{proof}[Sketch of proof]
Writing the $p$th power as a product of $p$ integrals, we can write
\begin{equation}\label{eq:rewriteUdp2}
\begin{aligned}
[f]_{\textup{U}^{d,p}(G)}^{2^d p} = \int_{G^{p+d}} \prod_{j=1}^p \prod_{\omega=(\omega_1,\ldots,\omega_d)\in\{0,1\}^d}
\!\!\!\!\!\!\!\!\mathcal{C}^{d-(\omega_1+\cdots+\omega_d)+j} f\bigl(x_j + \omega_1 h_1 + \cdots + \omega_d h_d\bigr) & \\[-2mm]
\dd\mu(x_1) \cdots \dd\mu(x_p) \dd\mu(h_1) \cdots \dd\mu(h_d) & ,
\end{aligned}
\end{equation}
and the desired estimate follows by successive applications of H\"{o}lder's inequality as in \cite[Appen.]{Shkredov23}.
\end{proof}

When $G$ is compact, one can also deduce the previous lemma from the hypergraph H\"{o}lder inequality in \cite{Hatami} for the complete multipartite hypergraph of type $(p,2,\ldots,2)$.

\begin{corollary}\label{cor:Udpnorms}
If $d\geq2$ and $p$ is a positive integer, then $f\mapsto[|f|]_{\textup{U}^{d,p}(G)}$ is a norm on the vector space of $\mu$-a.e.\@ equivalence classes of complex measurable functions such that $[|f|]_{\textup{U}^{d,p}(G)}<\infty$.
If $d\geq1$ and $p$ is an even positive integer, then $f\mapsto[f]_{\textup{U}^{d,p}(G)}$ is also a norm on the same vector space.
\end{corollary}

\begin{proof}[Sketch of proof]
We have $[|f|]_{\textup{U}^{d,p}(G)}=0$ if and only if $R_d |f|$ vanishes $\mu^d$-a.e.\@ on $G^d$, which, in turn, happens if and only if
\[ \bigl\| |f| \bigr\|_{\textup{U}^d(G)} = \bigl\| R_d |f| \,\bigr\|_{\textup{L}^1(G^d)}^{1/2^d} = 0. \]
By known properties of the usual Gowers norms this means that $f$ equals zero $\mu$-a.e.~on~$G$.
In the complex-valued case, the same conclusion follows from the definiteness of autocorrelation, applied successively in the difference variables.

The remaining nontrivial property is the triangle inequality. We apply \eqref{eq:rewriteUdp2} to $f_1+f_2$ (when $p$ is not even we use $|f_1+f_2|\leq|f_1|+|f_2|$ to reduce the claim to nonnegative functions). Expansion by multilinearity produces $2^{2^d p}$ terms. Applying Lemma~\ref{lm:UdpGCS} to each resulting term and summing the bounds yields
\[ [f_1+f_2]_{\textup{U}^{d,p}(G)}^{2^d p}
\leq\bigl([f_1]_{\textup{U}^{d,p}(G)}+[f_2]_{\textup{U}^{d,p}(G)}\bigr)^{2^d p}, \]
which proves the triangle inequality.
\end{proof}


\section{Non-integer values of \texorpdfstring{$p$}{p} do not lead to norms}
\label{sec:Udp_not_norms}

Here is a counterexample to the triangle inequality for $[|\cdot|]_{\textup{U}^{d,p}}$ when $p$ is not an integer.

\begin{proposition}
Consider the functions $f_1,f_2$ on $\T$ defined by
\begin{align*}
f_1(x) & := 1 + a\cos (2\pi x) + b\cos(2N\pi x), \\
f_2(x) & := 1 + a\cos (2\pi x) - b\cos(2N\pi x).
\end{align*}
For every $d\geq2$ and $p\in(0,\infty)\setminus\N$, there exist $N\in\N$ and $a,b>0$ such that $a+b<1$ and $f_1,f_2$ violate the triangle inequality for $[\,\cdot\,]_{\textup{U}^{d,p}}$ in the sense of \eqref{eq:triangle_fails}.
\end{proposition}

The counterexample is adapted from \cite{Boas62} and \cite[Sec.\,7.7,\,\S3]{Montgomery}; see also the applications of Pringsheim's theorem in \cite{Boas}.
Figure~\ref{fig:gowers_triangle}, produced in Mathematica \cite{Mathematica}, illustrates choices of the parameters $N,a,b$ for which \eqref{eq:triangle_fails} holds for ``most'' values of $p$ in the intervals $(1,2)$ and $(2,3)$, in the left and right panels respectively. The corresponding functions $f_1$ are shown in Figure~\ref{fig:gowers_triangle2}.

\begin{figure}
\begin{center}
\includegraphics[width=0.45\linewidth]{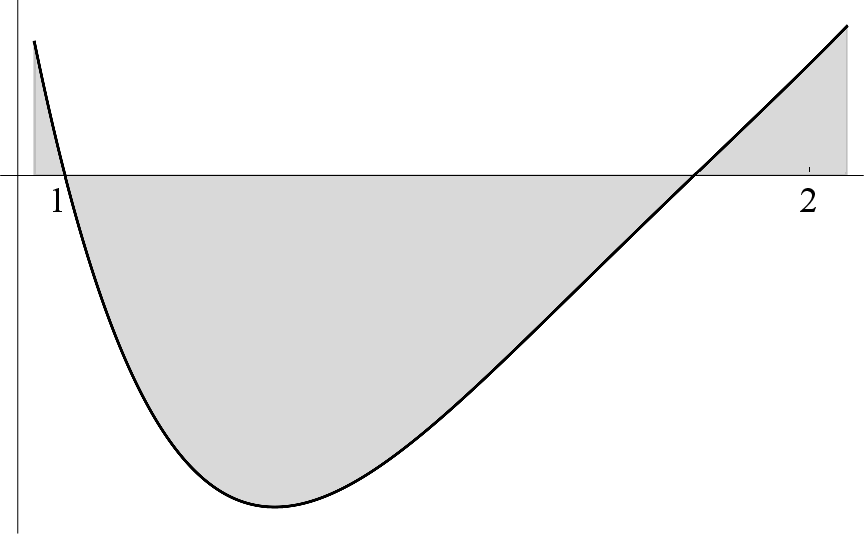}\hspace*{0.5cm}
\includegraphics[width=0.45\linewidth]{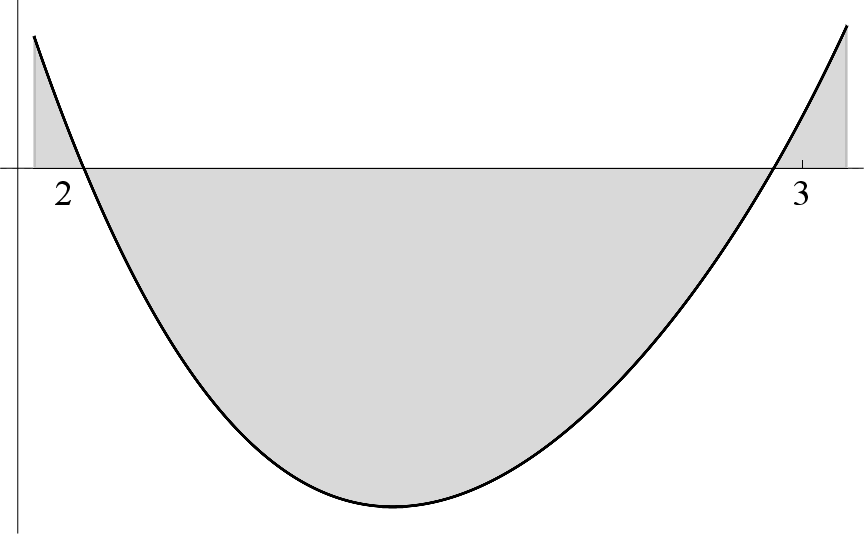}
\end{center}
\caption{Graphs of $p\mapsto [f_1]_{\textup{U}^{2,p}} + [f_2]_{\textup{U}^{2,p}} - [f_1+f_2]_{\textup{U}^{2,p}}$ for $N=4$, $a=24/25$, $b=3/100$ (left) and $N=6$, $a=199/200$, $b=1/500$ (right).}
\label{fig:gowers_triangle}
\end{figure}

\begin{figure}
\begin{center}
\includegraphics[width=0.45\linewidth]{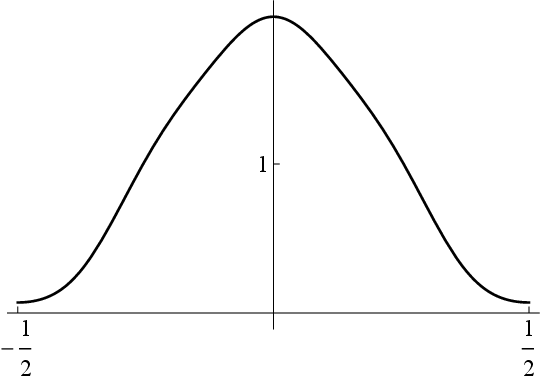}\hspace*{0.5cm}
\includegraphics[width=0.45\linewidth]{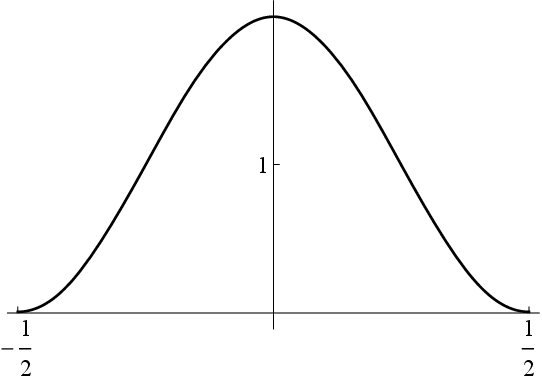}
\end{center}
\caption{Graphs of $f_1$ for the same choices of parameters as in Figure \ref{fig:gowers_triangle}.}
\label{fig:gowers_triangle2}
\end{figure}

\begin{proof}[Sketch of proof]
For fixed $N\in\N$ with $N>2^d$ and $a\in(0,1)$, it is not difficult to compute the following expansions in the small parameter $b>0$:
\[ [f_1]_{\textup{U}^{d,p}(\T)}^{p 2^d} = [f_2]_{\textup{U}^{d,p}(\T)}^{p 2^d} = [1+a\cos(2\pi\,\cdot)]_{\textup{U}^{d,p}(\T)}^{2^d p} + p b^2 \Lambda_N(a) + O_{d,p,N,a}^{b\to0}(b^4), \]
where
\begin{align*}
\Lambda_N(a) := \frac{1}{4} \sum_{\substack{\alpha,\beta\in\{0,1\}^d\\ \alpha\neq\beta}} \int_{\T^d}
& \bigl(R_d(1+a\cos(2\pi\,\cdot\,))(h)\bigr)^{p-1} \cos\bigl(2\pi N(\alpha-\beta)\cdot h\bigr) \\[-2mm]
& \biggl(\int_{\T} \prod_{\omega\in\{0,1\}^d\setminus\{\alpha,\beta\}} \bigl(1+a\cos(2\pi(x+\omega\cdot h))\bigr)\dd x\biggr) \dd h_1\cdots\dd h_d.
\end{align*}
The condition $N>2^d$ eliminates the odd powers of $b$.
Taking into account that
\[ [f_1+f_2]_{\textup{U}^{d,p}(\T)} = 2[1+a\cos(2\pi\,\cdot\,)]_{\textup{U}^{d,p}(\T)}, \]
we see that we only need to ensure $\Lambda_N(a)<0$, since any sufficiently small $b>0$ will then be a good choice.

Next, for a fixed $N$, one can compute the asymptotic expansion of $\Lambda_N(a)$ for small $a>0$:
\begin{equation}\label{eq:Lambdacoeff}
\Lambda_N(a) = \frac{a^{2N}}{2^{2N+1}} c_N + O_{d,p,N}^{a\to0}(a^{2N+2}),
\end{equation}
where the numbers $c_n$ are defined as follows.
For a nonnegative integer $s$, a point $y\in\T^s$ and a tuple $\eta_0\in\{0,1\}^s$, we introduce the polynomials
\begin{align*}
Q_{s,y}(z) & := \sum_{j=0}^{2^s} \biggl| \sum_{\substack{E\subseteq\{0,1\}^s\\ |E|=j}} \exp\Bigl(2\pi i\sum_{\eta\in E}\eta\cdot y\Bigr) \biggr|^2 z^j, \\
S_{s,y,\eta_0}(z) & := \sum_{j=0}^{2^s-1} \biggl|\sum_{\substack{E\subseteq \{0,1\}^s\setminus\{\eta_0\}\\ |E|=j}} \exp\Bigl(2\pi i\sum_{\eta\in E}\eta\cdot y\Bigr) \biggr|^2 z^j.
\end{align*}
For any two distinct tuples $\alpha,\beta\in\{0,1\}^d$, put
\[ s_{\alpha,\beta} := d - \text{Hamming distance between } \alpha \text{ and } \beta. \]
Restrict $\alpha$ and $\beta$ to the coordinates on which they agree, and denote their common restriction by $\eta_{\alpha,\beta}\in\{0,1\}^{s_{\alpha,\beta}}$. Then $(c_n)_{n=0}^\infty$ is defined as the sequence of coefficients in the expansion about $z=0$ of the analytic function
\[ \mathcal{K}_{d,p}(z) := \frac{1}{2} \sum_{\substack{\alpha,\beta\in\{0,1\}^d\\ \alpha\neq\beta}} \int_{{\T}^{s_{\alpha,\beta}}} Q_{s_{\alpha,\beta},y}(z)^{p-1} S_{s_{\alpha,\beta},y,\eta_{\alpha,\beta}}(z)\, \dd y
= \sum_{n=0}^{\infty} c_n z^n. \]
When evaluating the power $(\cdot)^{p-1}$, we choose the branch near the origin that takes the value $1$ at $z=0$.
Integration over $\T^0$ means evaluation at its unique point.
In view of \eqref{eq:Lambdacoeff}, it remains only to show that $c_n<0$ for arbitrarily large $n$.

Suppose, for a contradiction, that $c_n\geq0$ for all sufficiently large indices $n$.
The coefficients of $Q_{s,y}(z)$ and $S_{s,y,\eta_0}(z)$ are nonnegative, and $Q_{s,y}(x)\geq1$ for $x\geq0$. Consequently, for each $x_0>0$ there is a complex neighbourhood of $x_0$ on which all the polynomials $Q_{s,y}(z)$ are nonzero. Choosing the branch that is positive on the positive real axis shows that $\mathcal{K}_{d,p}$ continues analytically through every point of that axis. After subtracting a polynomial, the remaining power series has nonnegative coefficients.
If its radius of convergence were finite, Pringsheim's theorem would force a singularity at a positive real number, contrary to the preceding analytic continuation. Hence that power series converges to an entire function, which has at most polynomial growth on the positive real axis. An entire power series with nonnegative coefficients and at most polynomial growth on the positive reals must itself be a polynomial.
Thus, it remains to rule out the possibility that $\mathcal{K}_{d,p}$ is a polynomial.

The leading coefficients of both $Q_{s,y}$ and $S_{s,y,\eta_0}$ are $1$, while their remaining coefficients are uniformly bounded in $y$. Consequently, for every positive integer $J$,
\[ Q_{s,y}(x)^{p-1}S_{s,y,\eta_0}(x) = x^{2^sp-1} \biggl(\sum_{j=0}^{J-1}A_{s,j,y,\eta_0}x^{-j}+O_{d,p,J}^{x\to+\infty}(x^{-J})\biggr), \]
uniformly in $y\in\T^s$, with $A_{s,0,y,\eta_0}=1$. Thus, if $\mathcal{K}_s$ denotes the sum of the terms in $\mathcal{K}_{d,p}$ for which $s_{\alpha,\beta}=s$, then
\[ \mathcal{K}_s(x) = x^{2^sp-1} \biggl( C_s+\sum_{j=1}^{J-1}C_{s,j}x^{-j}+O_{d,p,J}^{x\to+\infty}(x^{-J}) \biggr) \]
for every $J$, where $C_s>0$.
Since $p\not\in\N$, we can choose the largest $s_0\in\{0,1,\ldots,d-1\}$ such that $2^{s_0}p\not\in\Z$. For $s<s_0$ we have
\[ \mathcal{K}_s(x) = o_{d,p,s}^{x\to+\infty}(x^{2^{s_0}p-1}), \]
while for $s>s_0$ all powers in the preceding expansion are integers. Taking sufficiently many terms in those finitely many expansions and collecting the terms of order at least $x^{2^{s_0}p-1}$, we obtain a polynomial $P$ such that
\[ \mathcal{K}_{d,p}(x) - P(x) = C_{s_0} x^{2^{s_0}p-1} + o_{d,p}^{x\to+\infty}(x^{2^{s_0}p-1}). \]
The exponent of $x^{2^{s_0}p-1}$ is not an integer and $C_{s_0}>0$, so no polynomial can have the stated asymptotic behaviour and we arrive at a contradiction.

We conclude that $c_N<0$ for arbitrarily large $N$. Choose one such $N>2^d$. By \eqref{eq:Lambdacoeff} we have $\Lambda_N(a)<0$ for every sufficiently small $a>0$. After fixing one such $a\in(0,1)$, we choose $b>0$ sufficiently small that $a+b<1$ and the expansions in $b$ give \eqref{eq:triangle_fails}.
\end{proof}

\section*{Declaration of AI usage}

OpenAI's ChatGPT-5.6 Sol and ChatGPT-6 Astra were used for copy-editing and several rounds of proofreading of the finished manuscript. ChatGPT-5.6 Sol also wrote the Mathematica code used to produce the figures. The ideas, the proofs and the manuscript text are entirely the work of the authors.

We also note that ChatGPT-5.6 Sol was unable to prove Theorem~\ref{thm:mainlog} autonomously, even in the case $d=2$ in about six hours of thinking (see \href{https://chatgpt.com/share/6a9459ee-5770-83eb-97a1-571db00c55d5}{chat \#1}). After the manuscript had been finalised, we repeated the experiment with ChatGPT-6 Pro, which was able to solve the case $d=2$ of the problem (see \href{https://chatgpt.com/share/6a9bc3d4-e28c-83ed-95a2-25eab599a65d}{chat \#2}). However, its solution was very different from ours and used the Kullback--Leibler divergence and Pinsker's inequality. This approach did not fit naturally into our manuscript, but we wanted to record it for a possible future project, in case it proves more effective for determining an explicit constant on the real line.


\section*{Acknowledgements and funding}

Much of this collaboration took place in September and October 2024. The authors are grateful to the Inter-University Centre in Dubrovnik and the University of Zagreb Faculty of Science for their hospitality. The authors thank Jonathan Bennett for several useful comments on the literature. V.\,K. would also like to thank Aleksandar Bulj for a useful accidental discussion on the Hardy--Littlewood majorant problem.

V.\,K. was supported in part by the Croatian Science Foundation under the project HRZZ-IP-2022-10-5116 (FANAP) and in part by the European Union -- NextGenerationEU through the National Recovery and Resilience Plan 2021--2026, via an institutional grant from the University of Zagreb Faculty of Science, IK IA 1.1.3, \emph{Impact4Math}.
K.\,R. was supported by the MICINNU grants CEX2023-001347-S, PID2021-124195NB-C33 and PID2024-158664NB-C22.


\bibliographystyle{plainurl}
\bibliography{gen_Gow_fun}

\end{document}